\documentclass[11pt, reqno]{article}

\usepackage{fullpage}
\usepackage{enumerate}
\usepackage{setspace}

\usepackage{amsmath,amsfonts,amssymb,graphics,amsthm, comment}
\usepackage{hyperref}
\hypersetup{
    colorlinks=true,
    linkcolor=blue,
    citecolor=red,
    urlcolor=blue,
    pdfborder={0 0 0}
}

\usepackage[T1]{fontenc}
\usepackage{mathabx,graphicx}

\usepackage{graphicx}
\usepackage{xcolor}
\usepackage{bbm}
\usepackage{wrapfig} 

\usepackage[font=sf, labelfont={sf,bf}, margin=1cm]{caption}

\usepackage{cleveref}
  \crefname{theorem}{Theorem}{Theorems}
  \crefname{thm}{Theorem}{Theorems}
  \crefname{lemma}{Lemma}{Lemmas}
  \crefname{lem}{Lemma}{Lemmas}
  \crefname{remark}{Remark}{Remarks}
  \crefname{prop}{Proposition}{Propositions}
\crefname{notation}{Notation}{Notations}
\crefname{claim}{Claim}{Claims}
  \crefname{defn}{Definition}{Definitions}
  \crefname{corollary}{Corollary}{Corollaries}
  \crefname{section}{Section}{Sections}
  \crefname{figure}{Figure}{Figures}
    \crefname{assumption}{Assumption}{Assumptions}

\newtheorem{thm}{Theorem}[section]

\newtheorem{lemma}[thm]{Lemma}
\newtheorem{corollary}[thm]{Corollary}
\newtheorem{prop}[thm]{Proposition}
\newtheorem{defn}[thm]{Definition}

\newtheorem{problem}[thm]{Open Problem}
\newtheorem{assumption}[thm]{Assumptions}
\numberwithin{equation}{section}

\theoremstyle{definition}
\newtheorem{remark}[thm]{Remark}

\def \b {\mathsf{bias}}

\def \eqd {\stackrel{(d)}{=}}
\def \min {\text{min}}

\def \1  {(1)}
\def \2 {(2)}
\def \ii {(i)}

\def\cI{\mathcal{I}}

\def \ve {\varepsilon}

\def\P{\mathbb{P}}
\def\E{\mathbb{E}}
\def\C{\mathbb{C}}
\def\R{\mathbb{R}}

\def\N{\mathbb{N}}
\def\D{\mathbb{D}}

\def  \p- {p\textunderscore}

\def\eps{\varepsilon}

\def\e{\text{e}}
\def\ph{\varphi}
\def \b {\boldsymbol}

\DeclareMathOperator{\gff}{GFF}

\DeclareMathOperator{\var}{Var}

\def  \ve {\varepsilon}
\DeclareMathOperator{\supp}{Support}
\DeclareMathOperator{\Var}{Var}

\newcommand{\notet}[1]{{\color{red}{[#1]}}}

\title{A characterisation of the Gaussian free field}
\author{Nathana\"el Berestycki\thanks{Supported in part by EPSRC grants EP/L018896/1 and EP/I03372X/1} \and Ellen Powell \thanks{Supported by EPSRC grant EP/H023348/1 and NCCR SwissMAP} \and Gourab Ray\thanks{Supported in part by EPSRC grant EP/I03372X/1 and by University of Victoria start-up 10000-27458.}}

\begin{document}

\maketitle
\begin{abstract}
We prove that a random distribution in two dimensions which is conformally invariant and satisfies a natural domain Markov property is a multiple of the Gaussian free field. This result holds subject only to a fourth moment assumption.
\end{abstract}

\section{Introduction}

\subsection{Setup and main result}

 The \textbf{Gaussian free field} (abbreviated GFF)
 has emerged in recent years as an object of central importance in probability theory. In two dimensions in particular, the GFF is conjectured (and in many cases proved) to arise as a universal scaling limit from a broad range of models, including the Ginzburg--Landau $\nabla \varphi$ interface model (\cite{GOS, NS, MillerGL}), the height function associated to planar domino tilings and the dimer model (\cite{Kenyon_GFF, dubedat_torsion, BLRdimers,BLRtorus,DubedatGheissari,Li}), and the characteristic polynomial of random matrices (\cite{cue,RiderVirag,gue}). It also plays a crucial role in the mathematically rigourous description of Liouville quantum gravity; see in particular \cite{DKRV, DOZZ} and \cite{DuplantierMillerSheffield} for some recent major developments (we refer to \cite{polyakovstrings} for the original physics paper). Note that the interpretations of Liouville quantum gravity in the references above are slightly different from one another, and are in fact more closely related to the GFF with Neumann boundary conditions than the GFF with Dirichlet boundary conditions treated in this paper.

 As a canonical random distribution enjoying conformal invariance and a domain Markov property, the GFF is also intimately linked to the Schramm--Loewner Evolution (SLE). In particular SLE$_4$ and related curves can be viewed as level lines of the GFF (\cite{SchrammSheffield, SchrammSheffieldDiscrete, Dubedat, PowellWu}). In fact, this connection played an important role in the approach to Liouville quantum gravity developed in \cite{DuplantierSheffield, DuplantierMillerSheffield, zipper, LQGandTBMI} (see also \cite{LQGNotes} for an introduction).

It is natural to seek an axiomatic characterisation of the GFF which could explain this ubiquity. In the present article we propose one such characterisation, in the spirit of Schramm's celebrated characterisation of SLE as the unique family of conformally invariant laws on random curves satisfying a domain Markov property \cite{schrammcharacterisation}. 

As the GFF is a random distribution (and not a random function) we will need to pay attention to the measure-theoretic formulation of the problem. We start by introducing some notations.
Let $D$ be a simply connected domain and let $C_c^\infty(D)$ be the space of smooth functions that are compactly supported in $D$ (the space of so-called \textbf{test functions}). We equip it with the topology such that $\phi_n\to 0$ if and only if there is some $M\Subset D$ containing the supports of all the $\phi_n$, and all the derivatives of $\phi_n$ converge uniformly to $0$. (Here and in the rest of the paper, the notation $M\Subset D$ means that the closure of $M$ is compact and contained in the open set $D$.)  For any two test functions $\phi_1,\phi_2$, we define $$(\phi_1,\phi_2) := \int \phi_1(z) \phi_2(z)dz,$$ and for any test function $\phi$ we call $(\phi,1)$ the \emph{mass} of $\phi$.

In order to avoid discussing random variables taking values in the space of distributions in $D$ we take the simpler and more general point of view that we have a stochastic process $h^D = (h^D_\phi)_{\phi \in C^\infty_c(D) }$ indexed by test functions and which is linear in $\phi$: that is, for any $\lambda, \mu \in \R$ and $\phi, \phi'\in C_c^\infty(D)$, $$h^D_{\lambda \phi + \mu \phi'} = \lambda h^D_\phi + \mu h^D_{\phi'},$$ almost surely. We then write, with an abuse of notation, $(h^D, \phi) = h^D_\phi$ for $\phi \in C_c^\infty(D)$. We call $\Gamma^D $ the law of the stochastic process $(h^D_\phi)_{\phi \in C_c^\infty(D)}$. Thus $\Gamma^D$ is a probability distribution on $\R^{C^\infty_c(D)}$ equipped with the product topology; by Kolmogorov's extension theorem $\Gamma^D$ is characterised by its consistent finite-dimensional distributions, i.e., by the joint law of $(h^D, \phi_1), \ldots, (h^D, \phi_k)$ for any $k \ge 1$ and any $\phi_1, \ldots, \phi_k \in C_c^\infty(D)$.


\medskip Suppose that $\Gamma:=\{ \Gamma^D\}_{D \subset \C}$ is a collection of such measures, where $D\subset \C$ ranges over all simply connected proper domains and $\Gamma^D$ is as above for each simply connected proper domain $D$. We will always denote by $\mathbb{D}$ the unit disc of the complex plane. We now state our assumptions:

\begin{assumption} Let $D \subset \C$ be a proper simply connected open domain, and let $h^D$ be a sample from $\Gamma^D$. We assume the following:

\label{ass:ci_dmp}
\begin{enumerate}[(i)]

\item \textbf{(Moments, stochastic continuity.)} For every $\phi \in C_c^\infty (D)$,
$$\E[(h^D,\phi)]=0 \;\; \text{and} \;\; \E[(h^D,\phi)^4]<\infty. $$
Moreover,  there exists a continuous bilinear form $K_2^D$ on $C_c^\infty(D)\times C_c^\infty(D)$  such that
	\[ \E[(h^D,\phi)(h^D,\phi')]=K_2^D(\phi,\phi'), \quad \quad \phi, \phi' \in C_c^\infty(D).\]

\item \textbf{(Dirichlet boundary conditions)} Suppose that $(f_n)_{n \ge 1}$ is a sequence of nonnegative, radially symmetric functions in $C_c^\infty(\D)$, with uniformly bounded mass and such that for every $M \Subset \D$,  $\supp(f_n) \cap M = \emptyset$ for all large enough $n$. Then we have $\Var((h^\D,f_n)) \to 0 \text{ as } n \to \infty.$

\item \textbf{(Conformal invariance.)} Let $f: D \to D'$ be a bijective conformal map. Then
$
\Gamma^{D}  = \Gamma^{D'}\circ f,
$
where $\Gamma^{D'} \circ f$ is the law of the stochastic process $(h^{D'}, |(f^{-1})'|^2 (\phi \circ f^{-1}))_{\phi \in C_c^\infty(D)}$.

\item \textbf{(Domain Markov property)}. Suppose $D' \subset D$ is a simply connected Jordan domain. Then  we can decompose
$
h^D=  h^{D'}_D+\ph_D^{D'}
$
where:
\begin{itemize}
	\item $ h^{D'}_D$ is independent of $\ph_D^{D'}$;
	\item $(\ph_D^{D'},\phi)_{\phi\in C_c^\infty(D)}$ is a stochastic process indexed by $C_c^\infty(D)$ that is a.s. linear in $\phi$ and such that $(\ph_D^{D'},\phi)_{\phi\in C_c^\infty(D')}$ a.s. corresponds to integrating against a harmonic function in $D'$;
	\item $((h^{D'}_D,\phi))_{\phi\in C_c^\infty(D)}$ is a stochastic process indexed by $C_c^\infty(D)$, such that $(h^{D'}_D,\phi)_{\phi\in C_c^\infty(D')}$ has law $\Gamma^{D'}$ and $(h^{D'}_D,\phi)=0$ a.s. for any $\phi$ with $\supp(\phi)\subset D\setminus D'$.
\end{itemize}
\end{enumerate}
\end{assumption}

\begin{remark}
	\label{rmk::markov_zero_outside}
	Note that in the domain Markov property, we have (by linearity) that if $D'\subset D$ is simply connected, and $\phi_1=\phi_2$ on $D'$, then $(h^{D'}_D,\phi_1)=(h^{D'}_D,\phi_2)$ almost surely.
\end{remark}

	When we discuss the domain Markov property later in the paper, we will often simply say that $$ ``\, \ph_D^{D'} \text{ is harmonic in } D'\, , \, h_D^{D'} \text{ is } 0 \text{ in } D\setminus D' \text{ and }\, h_D^{D'}\overset{(d)}{=}h^{D'}\text{ in } D'\,".$$ These statements should be interpreted as described rigorously in Assumptions \ref{ass:ci_dmp}.

\begin{remark}
	The finite fourth moment condition implies, in particular, that there exists a quadrilinear form $K_4^D$ on $(C_c^\infty(D))^{\otimes 4}$ such that for every $\phi_1,\cdots, \phi_4 \in C_c^\infty(D)$,
	$$ \E[(h^D,\phi_1)(h^D,\phi_2)(h^D,\phi_3)(h^D,\phi_4)]=K_4^D(\phi_1,\phi_2,\phi_3,\phi_4).$$
\end{remark}

\begin{lemma}\label{lem:unicity_decomposition}
The assumption of zero boundary conditions implies that the domain Markov decomposition from (iv) is unique.
\end{lemma}
\begin{proof}
Suppose that we have two such decompositions:
\begin{equation}\label{eqn::DMP_uniqueness}
h^D = h^{D'}_D+\ph_D^{D'} =  \tilde{h}^{D'}_D+\tilde{\ph}_D^{D'}.
\end{equation}
Suppose that we have two such decompositions:
\begin{equation}\label{eqn::DMP_uniqueness}
h^D = h^{D'}_D+\ph_D^{D'} =  \tilde{h}^{D'}_D+\tilde{\ph}_D^{D'}.
\end{equation}
Pick any $z\in D'$ and let $F:D'\to \D$ be a conformal map that sends $z$ to $0$. Further, let $(f_n)_{n\ge 1}$ be a sequence of nonnegative radially symmetric, mass one functions in $C_c^\infty(\D)$, that are eventually supported outside any $K\Subset \D$, and set $g_n := |F'|^2 (f_n \circ F)$ for each $n$. Then the assumption of Dirichlet boundary conditions plus conformal invariance implies that $(h^{D'}_D-\tilde{h}^{D'}_D,g_n )\to 0$ in probability as $n\to \infty$. In turn, by (\ref{eqn::DMP_uniqueness}), this means that $(\ph_D^{D'}-\tilde{\ph}_D^{D'},g_n) \to 0$ in probability.

However, since $(\ph_D^{D'} - \tilde{\ph}_D^{D'})$ restricted to $D'$ is a.s. equal to a harmonic function, and since the $f_n$'s are radially symmetric with mass one, we have
\[ (\ph_D^{D'}-\tilde{\ph}_D^{D'},g_n) =((\ph_D^{D'}-\tilde{\ph}_D^{D'})\circ F^{-1}, f_n)=(\ph_D^{D'}-\tilde{\ph}_D^{D'})\circ F^{-1}(0)=\ph_D^{D'}(z)-\tilde{\ph}_D^{D'}(z)\] for every $n$. This implies that for each fixed $z\in D'$, $\ph_D^{D'}(z)=\tilde{\ph}_D^{D'}(z)$ a.s. Applying this to a countable dense subset of $z\in D'$, together with the fact that $h^D=\ph_D^{D'}=\tilde{\ph}_D^{D'}$ a.s. outside of $D'$, see Remark \ref{rmk::markov_zero_outside}, then implies that $\ph_D^{D'}$ and $\tilde \ph_D^{D'}$ are a.s. equal as stochastic processes indexed by $C_c^\infty(D)$.
\end{proof}


\begin{defn} \label{def::gff} A mean zero Gaussian free field $h_{\gff} =h^D_{\gff}  $
with zero boundary conditions is a stochastic process indexed by test functions $(h_{\gff}, \ph)_{\ph \in C_c^\infty(D)}$ such that:

\begin{itemize}
\item $h_{\gff}$ is a centered Gaussian field; for any $n\ge 1$ and any set of test functions $\phi_1,\cdots, \phi_n \in C_c^\infty(D)$, $((h_{\gff},\phi_1),\cdots, (h_{\gff},\phi_n))$ is a Gaussian random vector with mean ${\mathbf{0}}$;
\item for any two test functions $\phi_1,\phi_2 \in C_c^\infty(D)$,
$$
\E[(h_{\gff},\phi_1) , (h_{\gff},\phi_2)] = \int_{D} G^D(z,w) \phi_1(z)\phi_2(w)dzdw
$$
where $G^D$ is the Green's function with Dirichlet boundary conditions on $D$.
\end{itemize}
\end{defn}

It is well known and easy to check (see e.g. \cite{LQGNotes}) that Assumptions \ref{ass:ci_dmp} are satisfied for the collection of laws $\{\Gamma^D_{\gff}; D \subset \C\}$ obtained by considering the GFF, $h^D_{\gff}$, in proper simply connected domains. More generally any multiple of the GFF $\alpha h^D_{\gff}$ (with $\alpha \in \R$) will verify these assumptions. (In fact, the boundary conditions satisfied by the GFF are much stronger than what we assume: it is not just the average value of the GFF on the unit circle which is zero, but, e.g., the average value on any open arc of the unit circle.) The main result of this paper is the following converse:

\begin{thm}\label{thm::characterisation_gff}
Suppose the collection of laws $\{\Gamma^D\}_{D\subset \mathbb{C}}$ satisfy Assumptions \ref{ass:ci_dmp} and let $h^D$ be a sample from $\Gamma^D$. Then there exists $\alpha \in \R$ such that $h^D = \alpha h_{\gff}^D$ in law, as stochastic processes.
\end{thm}


\begin{remark}
Given the close relationship between the GFF and SLE, it is natural to  wonder if the characterisation Theorem \ref{thm::characterisation_gff} could be deduced from Schramm's celebrated characterisation (and discovery) of SLE curves \cite{schrammcharacterisation}. Perhaps if one is also given an appropriately defined notion of \emph{local sets} in addition to the field (see \cite{SchrammSheffield,BTLS}), one could identify these local sets as SLE type curves with some unknown parameter. However, even this would not be sufficient to identify the field as the GFF. Indeed, note that the CLE$_\kappa$ nesting fields (\cite{NestingFieldCLE}) provide examples of conformally invariant random fields coupled with SLE-type local sets, yet are only believed to be Gaussian in the case $\kappa = 4$.
\end{remark}

\subsection{Role of our assumptions}
We take a moment to discuss the role of our assumptions. The fundamental assumptions of Theorem \ref{thm::characterisation_gff} are (ii), (iii) and (iv) which cannot be dispensed with. To see that they are necessary, the reader might consider the following two examples:
\begin{itemize}
  \item The magnetisation field in the critical Ising model (\cite{CamiaGarbanNewman1, CamiaGarbanNewman2});
  \item The CLE$_\kappa$ nesting field (\cite{NestingFieldCLE}).
\end{itemize}

In both these examples, conformal invariance (or at least conformal covariance) and even a form of domain Markov property (but not exactly the one formulated here) hold; yet neither of these are the GFF (except in the second case when $\kappa = 4$). These two examples are the kind of possible counterexamples to keep in mind when considering Theorem \ref{thm::characterisation_gff} or possible variants.

\medskip The role of Assumption (i) however is more technical and is instead the result of a choice and/or limitations of our proof.

We do not know whether a fourth moment assumption is necessary. Our use of this assumption is to rule out by Kolmogorov estimates the possibility of Poissonian-type jumps. To explain the problem, the reader might think of the following rough analogy: if a centered process has independent and stationary increments, it does not follow that it is Brownian motion even if it has finite second moment; for instance, $(N_t - t)_{t \ge 0}$, where $N_t$ is a standard Poisson process satisfies these assumptions. 
See the section on open problems for more discussion.

Regarding the assumption of stochastic continuity, we point out that $(\phi, \phi') \mapsto K(\phi, \phi') = \E[ (h^D, \phi) (h^D, \phi')]$ is clearly a bilinear map. So the assumption we make is simply that this map is jointly continuous.  Another way to rephrase this assumption is to say that $\varphi \mapsto (h^D, \varphi)$ is continuous in $L^2(\P)$ (referred to as \textbf{stochastic continuity} by some authors), which seems quite basic.

\subsection{The one-dimensional case}

In one-dimension, the zero boundary GFF reduces to a Brownian bridge (see e.g. Sheffield \cite{Sheffield}). However, even in this classical setup it seems that a characterisation of the Brownian bridge along the lines we have proposed in Theorem \ref{thm::characterisation_gff} was not known. Of course we need to pay some attention to the assumptions here, since it is not the case that a GFF is scale-invariant in dimension $d\ne 2$. 
Instead, the Brownian bridge enjoys Brownian scaling.

 Let $\cI$ be the space of all closed, bounded intervals of $\R$ and assume that for each $I\in \cI$ we have a stochastic process $X^I=(X^I(t))_{t \in I }$ indexed by the points of $I$. We let $\mu^I$ be the law of the stochastic process $(X^I(t))_{t\in I}$, so that $\mu^I$ is a probability distribution on $\R^{I}$ equipped with the product topology. Similarly to the two-dimensional case, by Kolmogorov's extension theorem, $\mu^I$ is characterised by its consistent finite-dimensional distributions, i.e., by the joint law of $X^I(t_1), \ldots, X^I(t_k)$ for any $k \ge 1$ and any $t_1,\cdots, t_k\in I$.

\begin{assumption}\label{ass:BB}
We make the following assumptions.

\begin{enumerate}[(i)]

 \item {\textbf{(Tails)} For each $I$ and $t\in I$, $\E[\log^+|X^I(t)|]<\infty$.}

 \item { \textbf{(Stochastic continuity)} For each $I$ the process $(X^I(t))_{t\in I}$ is stochastically continuous: that is, $\lim_{s\to t} \P(|X^I(t)-X^I(s)|>\eps)=0$ for every $\eps>0$.}

\item \textbf{(Zero boundary condition.)} For each interval $I =[a,b]$, $X^{I}(a)=X^I(b) =0$.

\item \textbf{(Domain Markov property.)} For each $I'  =[a,b] \subset I$, conditioned on $(X^I(t))_{t \in I \setminus I'}$,  the law of $(X^I(s))_{ s \in I'}$ is the same as
$$
L(s) + \tilde X^{I'}(s) \quad ; \quad s \in I'
$$
where $L(s)$ is a linear function interpolating between $X^I(a)$ and $X^I(b)$ and $\tilde X^{I'}$ is an independent copy of $X^{I'}$.

\item \textbf{(Translation invariance and scaling.)} For any $a\in \R, c>0$
$$
(X^{I-a}(t-a))_{t \in I}   \eqd (X^{I}(t))_{t \in I}
$$
 and
 $$
 (\frac{1}{\sqrt{c}}X^{cI}(ct))_{t \in I} \eqd (X^I(t))_{t \in I}.
 $$
\end{enumerate}
\end{assumption}

Our result in this case is as follows:

\begin{thm}
  \label{thm:BB} Subject to Assumptions \ref{ass:BB}, a sample $X^I$ has the law of a multiple $\sigma$ of a Brownian bridge on the interval $I$, from zero to zero.
\end{thm}

Interestingly, the proof in this case is substantially different from the planar case, and relies on stochastic calculus arguments. The definition in \cref{ass:BB} is reminiscent of the classical notion of \textbf{harness} in one dimension: roughly speaking, a square integrable continuous process such that conditionally on the process outside of any interval, the process inside has an expectation which is the linear interpolation of the data outside. If such a process is defined on the entire nonnegative halfline, then Williams \cite{Williams_harness} proved that a harness is a multiple of Brownian motion plus drift; see Mansuy and Yor \cite{MansuyYor} for a survey and extensions.  Theorem \ref{thm:BB} may therefore be seen as a generalisation of Williams' result to the case where the underlying domain is bounded, without assuming continuity and assuming only logarithmic tails (but assuming more in terms of the domain Markov property). To our knowledge, this result has not been previously considered in the literature.


\subsection{Outline}

We now summarise the structure of the proof of the main result (Theorem \ref{thm::characterisation_gff}) and explain the organisation of the paper.

Our first goal is to make sense of circle averages of the field, which exist as a result of the domain Markov property, conformal invariance and zero boundary condition (Section \ref{sec::circle_avg}). These circle averages can then fairly easily be seen to give rise to a two-point function $\tilde K_2(z_1, z_2)$ (Section \ref{sec::harm_avg}). Intuitively, the bilinear form $K_2$ in the assumption is simply the integral operator associated with this two-point function, but we do not need to establish this immediately (instead, it will follow from some estimates obtained later; see Lemma \ref{lemma::circ_avg_good_approx}). In Section \ref{sec::cor_estimates} we establish \emph{a priori} logarithmic bounds on the two-point (and four-point) functions which are needed to control errors later on. The Markov property and conformal invariance are easily seen to imply that the two point function is harmonic off the diagonal (Section \ref{sec:harmonic}). This point of view culminates in Section \ref{sec:twopointGreen}, where it is shown that the two point function is necessarily a multiple of the Green's function. (Intuitively, we rely on the fact that the Green's function is characterised by harmonicity and logarithmic divergence on the diagonal, though our proof exploits an essentially equivalent but slightly shorter route). At this point we still have not made use of our fourth moment assumption.

To conclude it remains to show that the field is Gaussian in the sense that any test function $(h, \ph)$ is a centered Gaussian random variable. This is the subject of Section \ref{sec::gaussianity} and is the most delicate and interesting part of the argument. The Gaussianity comes from an application of L\'evy's characterisation of Brownian motion, or more precisely, from the Dubins--Schwarz theorem. For this we need a certain process to be a continuous martingale, and it is only here that our fourth moment assumption is required: we use it in combination with a Kolmogorov continuity criterion and a deformation argument exploiting the form of a well-chosen family of conformal maps to prove continuity. The arguments are combined in Section \ref{sec:Concl} to conclude the proof of Theorem \ref{thm::characterisation_gff}. Finally, the last section (Section \ref{sec:BB}) gives a proof in the one-dimensional case (Theorem \ref{thm:BB}) using stochastic calculus techniques. The paper concludes with a discussion of open problems in Section \ref{S:problems}.

\paragraph{Acknowledgements:} We would like to thank Omer Angel, Juhan Aru, Chris Burdzy, Benoit Laslier, Soumik Pal and Steffen Rohde for several useful discussions. We thank Scott Sheffield for very useful comments on a preliminary draft and correcting a small mistake.  NB is especially grateful to Juhan Aru for raising the question of characterisation of the Gaussian free field with him, which eventually led to this paper. Finally, we would like to thank the associate editor and anonymous referee for many suggestions that helped us to improve the presentation of the paper.

\section{Two-point and four-point functions}

To begin with, we make sense of \emph{circle averages} of our field. These will play a key role in the proof of Theorem \ref{thm::characterisation_gff}, as we will be able to identify the law of the circle average process around a point with a one-dimensional Brownian motion.

In fact, we will define something more general. Let $\gamma$ be the boundary of a Jordan domain $D'\subseteq D$. We will, given $z\in D'$, define the harmonic average (as seen from $z$) of $h$ on $\gamma$ and will denote this average by $(h^D,\rho_z^\gamma)$. Note that since $h$ can only be tested a priori against smooth functions, and therefore not necessarily against the harmonic measure on $\gamma$, this is a slight abuse of notation. We will define the average in two equivalent ways: through an approximation procedure, and using the domain Markov property of the field.

\subsection{Circle average} \label{sec::circle_avg}
 Let $D$ be a simply connected domain such that $\mathbb{D}\subseteq D$ where $\D$ is the unit disc. We will first try to define $(h^{D}, \rho_0^{\partial \D} )$ as described above. To this end, let $\tilde\psi_0^{\delta}$ be a smooth radially symmetric function taking values in $[0,1]$, that is equal to $1$ on $A:=\{z:1-\delta\leq |z|\leq 1-\delta/2\}$ and is equal to $0$ outside of the $\delta/10$ neighbourhood of the annulus $A$. Let $\psi_0^{\delta}= \tilde{\psi}_0^\delta/\int \tilde{\psi}_0^\delta$. Then for all $\delta\in [0,1]$, since $\psi_0^\delta \in C_c^\infty(D)$, the quantity $(h^D,\psi_0^{\delta})$ is well defined. We will take a limit as $\delta\to 0$ to define the circle average (the precise definition of $\psi_0^{\delta}$ does not matter, as will become clear from the proof).

\begin{lemma}\label{lemma::circle_average_defn}
	$$\lim_{\delta \to 0}(h^D, \psi_0^{\delta}) =: (h^D,\rho_0^{\partial \D})$$ exists in probability and in $L^2(\mathbb{P})$. Moreover,
	$$ (h^D,\rho_0^{\partial \D})=\varphi_D^{\D}(0)$$
	where $h^D=h_D^{\D}+\varphi_D^{\D}$ is the domain Markov decomposition of $h^D$ in $\D$ described in Assumptions \ref{ass:ci_dmp}.
\end{lemma}

\begin{proof}
	We write $(h^D,\psi_0^{\delta})=(h_D^{\D},\psi_0^{\delta})+(\varphi_D^{\D}, \psi_0^{\delta})$ using the domain Markov decomposition. Note that because $\psi_0^\delta$ is radially symmetric with mass $1$, and is supported strictly inside $\D$ for each $\delta$, by harmonicity $(\varphi_D^{\D},\psi_0^{\delta})$ must be constant and equal to $\varphi_D^\D(0)$. Thus, we need only show that
	$$\lim_{\delta \to 0 } \Var((h_D^{\D},\psi_0^{\delta}))=0.$$
	However this follows from the fact that $h^{\D}_D\overset{(d)}{=} h^\D$ has zero boundary conditions (see the definition in Assumptions \ref{ass:ci_dmp}), since for any $M\Subset \D$, $\psi_0^\delta$ is supported outside of $M$ for small enough $\delta$ and is radially symmetric. Note that the rate of convergence of the variance to $0$ is uniform in the choice of domain $D$.
\end{proof}

\begin{remark} We could have simply defined $(h^D,\rho_0^{\partial \D} ):=\varphi_D^{\D}(0)$ as above. The reason we use the definition in terms of limits is so that later we are able to estimate its moments.
\end{remark}

\subsection{Harmonic average}\label{sec::harm_avg}
Now, let $D'\subset D$ be a Jordan domain bounded by a curve $\gamma$. Given $z\in D'$, also let $f:D'\to \mathbb{D}$ be the unique conformal map sending $z\mapsto 0$ and with $f'(z)>0$. We define $$\hat{\psi}_z^\delta:= |f'|^2\, (\psi^\delta_0\circ f) $$ and then set
$$ (h^D,\rho_z^\gamma):= \lim_{\delta \to 0 } (h^D,\hat{\psi}_z^\delta)$$ which we know exists in $L^2$ and in probability by the same argument as in the proof of Lemma \ref{lemma::circle_average_defn} (note that by conformal invariance, $(h_D^{D'}, \hat{\psi}_z^\delta)$ is equal to $(h^\D, \psi_0^\delta)$ in law if $h^D=h_D^{D'}+\varphi_D^{D'}$ is the domain Markov decomposition of $h^D$ in $D'$).
 Again, we could have simply defined the harmonic average to be equal to $\varphi_D^{D'}(z)$.

 It is clear that the harmonic average is always a random variable with mean $0$. We record here another useful property:

\begin{lemma}\label{lemma::harm_av_monotone}
Suppose $D''\subset D'\subset D$ are Jordan domains and $z\in D''$. Then
$$ \E[(h^D,\rho_z^{\partial D''})^2]\geq \E[(h^D,\rho_z^{\partial D'})^2] \;\; \text{and} \;\; \E[(h^D,\rho_z^{\partial D''})^4]\geq \E[(h^D,\rho_z^{\partial D'})^4].$$
\end{lemma}
\begin{proof}
Let $h^D=h_D^{D'}+\varphi_D^{D'}$ according to the domain Markov decomposition of $h^D$ in $D'$. Then we have that $(h^D,\rho_z^{\partial D'})=\varphi_D^{D'}(z)$. We can also decompose $h_D^{D'}$ inside $D''$ as $h_D^{D'}=h_{D'}^{D''}+\varphi_{D'}^{D''}$, which means (by uniqueness of the decomposition) that $(h^D,\rho_z^{\partial D''})=\varphi_D^{D'}(z)+\varphi_{D'}^{D''}(z)$. By independence of $\varphi_D^{D'}(z)$ and $\varphi_{D'}^{D''}(z)$, and the fact that the harmonic average has mean $0$, the result follows.
\end{proof}

Later on in the proof we will also use some alternative approximations to $(h^D,\rho_z^\gamma)$, as different approximations will be useful in different contexts.

\subsection{Circle average field} Now consider a general simply connected domain $D$. By the above construction, we can define $$h_\eps^D(z):=(h^D,\rho_z^{\partial B_z(\eps)})=\ph_D^{B_z(\eps)}(z)$$ for all $z \in D$ and all $\eps$ small enough, depending on $z$. We call this the \emph{circle average field}. It will be important to know that this is a good approximation to our field when $\eps$ is small. To show this, we will first need the following lemma.

\begin{lemma}\label{lemma::kernel_function}
	For $z_1\ne z_2$ distinct points in $D$,
	$$\tilde{K}_2^D(z_1,z_2):=\lim_{\eps \to 0 }\E[h_\eps^D(z_1)h_\eps^D(z_2)]$$ exists.
	Moreover, for any $D_1,D_2
\subset D$ Jordan subdomains such that $D_1\cap D_2=\emptyset$ and $z_1\in D_1, z_2\in D_2$, we have $$\tilde{K}_2^D(z_1,z_2)=\mathbb{E}[\ph_D^{D_1}(z_1)\ph_D^{D_2}(z_2)]$$
	
\end{lemma}

\begin{proof}
	Let $D_1,D_2$ be as above and write, by the domain Markov property, $$h^D=h^{D_1}_D+\ph^{D_1}_D \text{ and } h^D=h^{D_2}_D+\ph^{D_2}_D,$$  so that for $ \ph=\ph^{D_1}_D-h^{D_2}_D=\ph^{D_2}_D-h^{D_1}_D$  we have
	\begin{equation}\label{eqn::2ball_dmp}
	h^D=h^{D_1}_D+h^{D_2}_D+\ph.
 \end{equation} By definition of the domain Markov property, we can see that $(\ph,\phi)_{\phi\in C_c^\infty(D)}$ is a stochastic process that a.s. corresponds to a harmonic function when restricted to $\phi$ in $ C_c^\infty(D_1)$ or $C_c^\infty(D_2)$: in fact, we have that $\ph=\ph_D^{D_1}$ in $D_1$ and $\ph=\ph_D^{D_2}$ in $D_2$.
Note that $h^{D_2}_D$ is measurable with respect to $\varphi_D^{D_1}$ by Remark \ref{rmk::markov_zero_outside} (and conversely with the indices 1 and 2 switched), so the three terms in (\ref{eqn::2ball_dmp}) are pairwise independent.

Now let $\eps<\min\{|z_1-z_2|/2, d(z_1,\partial D), d(z_2,\partial D)\}$. Choosing $D\supset D_1\supset B_\eps(z_1)$ and $D\supset D_2\supset B_\eps(z_2)$, this means (also using uniqueness of the domain Markov decomposition) that $\ph_D^{B_\eps(z_i)}=\ph+\ph_i$ for $i=1,2$ where $\ph,\ph_1,\ph_2$ are pairwise independent and centered (indeed, $\ph_1,\ph_2$ are measurable with respect to $h^{D_1}_D, h^{D_2}_D$ respectively). This implies that  $$\E[h^D_\eps(z_1)h^D_\eps(z_2)]=\E[(\ph+\ph_1)(z_1)(\ph+\ph_2)(z_2)]=\E[\ph(z_1)\ph(z_2)]=\E[\ph_D^{D_1}(z_1)\ph_D^{D_2}(z_2)].$$ Hence the limit as $\eps\to 0$ exists, and we also see that it is equal to $\E[\ph_D^{D_1}(z_1)\ph_D^{D_2}(z_2)]$ for any $D_1,D_2$ as in the statement of the Lemma.
	\end{proof}

	Similarly, we have the following:
\begin{lemma}\label{lemma::fourpointexpression}
	For $z_1,z_2,z_3,z_4$ be pairwise distinct points in $D$. Then
	$$\tilde{K}_4^D(z_1,z_2,z_3,z_4):=\lim_{\eps \to 0 }\E[h_\eps^D(z_1)h_\eps^D(z_2)h^D_\eps(z_3)h^D_{\eps}(z_4)]$$ exists. Moreover, for any $D_1,D_2,D_3,D_4
		\subset D$ Jordan subdomains such that $D_i\cap D_j=\emptyset$ for every $1\le i\ne j\le 4$ and $z_i\in D_i$ for $1\le i \le 4$, we have $$\tilde{K}_4^D(z_1,z_2,z_3,z_4)=\mathbb{E}[\ph_D^{D_1}(z_1)\ph_D^{D_2}(z_2)\ph_D^{D_3}(z_3)\ph_D^{D_4}(z_4)]$$
\end{lemma}
	
It will also be convenient in what follows to have an alternative, ``hands-on'' way of approximating $\tilde{K}_2^D$ and $\tilde{K}_4^D$, which corresponds to directy testing the field against smooth test functions (rather than using the slightly abstract notion of circle averages).
		
		\begin{defn}[Mollified field] \label{remark::K_alt_def} Let $\phi$ be a smooth radially symmetric function, supported in the unit disc, and with total mass $1$. Let $\phi^z_\eps(\cdot)=\eps^{-2}\phi(\frac{|\cdot-z|}{\eps})$ so that $\phi_{\eps}^z$ is smooth, radially symmetric, has  mass $1$, and is supported in $B_z(\eps)$. Define $\tilde{h}_\eps^D(z):=(h^D,\phi_{\eps}^z)$. Then by the domain Markov property again, we see that we can \emph{equivalently} write
	$$\tilde{K}_2^D(z_1,z_2)=\lim_{\eps \to 0 }\E[\tilde{h}_\eps^D(z_1)\tilde{h}_\eps^D(z_2)]$$ and
	$$\tilde{K}_4^D(z_1,z_2,z_3,z_4)=\lim_{\eps \to 0 }\E[\tilde h_\eps^D(z_1)\tilde h_\eps^D(z_2)\tilde h^D_\eps(z_3)\tilde h^D_{\eps}(z_4)].$$
	\end{defn}

Note that here we do not have $\tilde{h}_\eps^D(z)=\varphi_D^{B_z(\eps)}(z)$ for every $\eps$ (because $\phi_\eps^z$ has support inside $B_z(\eps)$), but we still have for small enough $\eps$ (depending on $z_1,z_2$) that $\E[\tilde{h}_\eps^D(z_1)\tilde{h}_\eps^D(z_2)]=\tilde{K}_2^D(z_1,z_2)$.
	
\subsection{Properties of the two point kernel}	\label{sec:harmonic}
We can now prove some of the important properties of our two point kernel $\tilde{K}_2^D$. Namely:

\begin{prop}[Harmonicity] \label{lemma::harmonic kernel} For any $x\in D$, $\tilde{K}_2^D(x,y)$, viewed as a function of $y$, is harmonic in $D\setminus \{x\}$. \end{prop}

\begin{prop}[Conformal invariance]
	\label{lemma:conformal_invariance}
Let $f:D\to f(D)$ be a conformal map. Then for any distinct $x\ne y$ in $D$
	$$\tilde{K}_2^D(x,y)=\tilde{K}_2^{f(D)}(f(x),f(y)).$$
\end{prop}

\begin{proof}[Proof of Proposition \ref{lemma::harmonic kernel}] This is a direct consequence of the following Lemma (Lemma \ref{lemma::harmonic_kernel_proof}) and \cite[\S 2.2, Theorem 3]{evans}. \end{proof}

\begin{lemma}\label{lemma::harmonic_kernel_proof} Fix $x\in D$. Then $\tilde{K}_2^D(x,\cdot)\in C^2(D\setminus \{x\})$. Moreover, for any $\eta>0$ and $y\in D$ such that $|x-y|\wedge d(y,\partial D)>\eta$:
	\begin{equation}\label{eqn::mean_value}\tilde{K}_2^D(x,y)=\frac{1}{|\partial B_y(\eta)|} \int_{\partial B_y(\eta)} \tilde{K}_2^D(x,w) \, dw.\end{equation}
	
\end{lemma}

\begin{proof} In fact, the first regularity statement follows from (\ref{eqn::mean_value}). Indeed, take $y\in D\setminus \{x\}$, pick $\eta<|x-y|\wedge d(y,\partial D)$, and also take a smooth radially symmetric function $\phi$ that has mass $1$ and is supported on $B_0(\eta/2)$. Set $f(z)=\int_{D} \tilde{K}^D_2(x,w)  \phi(z-w)dw $. Then $f\in C^\infty(U)$ where $U=B_y(\eta/2)$. Moreover, $f(z)=\tilde{K}_2^D(x,z)$ for $z\in U$ by \eqref{eqn::mean_value}. This implies that $f$ is twice continuously differentiable at $y$.
	
	Thus, we only need to prove (\ref{eqn::mean_value}). However, this follows almost immediately from the definition of $\tilde{K}_2^D$. Take $\eta$ and $y$ as in the statement, and pick $\eps>0,\eta'>\eta$ such that $B_x(\eps)$ and $B_y(\eta')$ lie entirely in $D$ and are disjoint.

Then by Lemma \ref{lemma::kernel_function} we have
		\[\tilde{K}_2^D(x,y)=\mathbb{E}[\ph_D^{B_x(\eps)}(x)\ph_D^{B_y(\eta')}(y)] \text{ and }  \tilde K_2^D(x,w)=\mathbb{E}[\ph_D^{B_x(\eps)}(x)\ph_D^{B_y(\eta')}(w)] \;\;\forall w\in \partial B_y(\eta).\]
		
This allows us to conclude, since
	$$\int_{\partial B_y(\eta)} \tilde{K}_2^D(x,w) dw= \int_{\partial B_y(\eta)} \mathbb{E}[\ph_D^{B_x(\eps)}(x)\ph_D^{B_y(\eta')}(w)] \, dw = \mathbb{E}[ \ph_D^{B_x(\eps)}(x) \int_{\partial  B_y(\eta)}\ph_D^{B_y(\eta')}(w)\, dw]$$ which by harmonicity of $\ph$ in $B_y(\eta')$ is equal to $|\partial  B_y(\eta)|$ times
	\begin{equation*}[\ph_D^{B_x(\eps)}(x)\ph_D^{B_y(\eta')}(y)]=\tilde{K}_2^D(x,y). \qedhere \end{equation*}
\end{proof}

\begin{proof}[Proof of Proposition \ref{lemma:conformal_invariance}]
 Let $D_x\ni x$, $D_y \ni y$ be two Jordan subdomains of $D$ such that $D_x\cap D_y=\emptyset$. Then we have
 	\[ \tilde K^{f(D)}_2(f(x),f(y))=\mathbb{E}[\ph_{f(D)}^{f(D_x)}(f(x))\ph_{f(D)}^{f(D_y)}(f(y))]=\mathbb{E}[\ph_D^{D_x}(x)\ph_D^{D_y}(y)]=\tilde K^D_2(x,y)\]
 where we have used Lemma \ref{lemma::kernel_function} in the first and final equalities, and conformal invariance of $h^D$ in the second.

\end{proof}

%

\subsection{Estimates on two- and four-point functions}\label{sec::cor_estimates}
Before we can proceed to identify the two-point function as the Green's function, we need to derive some bounds on $\tilde{K}_2^D$ and $\tilde{K}_4^D$. For any set of pairwise distinct points $z_1,\ldots, z_k \in D$, we define
\begin{equation}
R(z_i;z_1,\ldots,z_k) := \min_{j \neq i} |z_i - z_j|  \wedge R(z_i,D)/10 \label{eq:rzw}
\end{equation}
where $R(z,D)$ is the conformal radius of $z$ in the domain $D$. We also set
	\begin{align*}
	l_2 (z,w)^2 &:= \log\left(\frac{R(z,D)}{R(z;z,w)}\right) \log \left(\frac{R(w,D)}{R(w;z,w)}\right)\\
	l_4(z_1,\ldots,z_4)^4 &:= \prod_{i=1}^4 \left[\log^2 \left(\frac{R(z_i,D)}{R(z_i;z_1,\ldots,z_4)}\right)+ \log \left(\frac{R(z_i,D)}{R(z_i;z_1,\ldots,z_4) }\right)\right].
	\end{align*}
The following logarithmic bounds are the main results of this section. We will use these repeatedly in the sequel, in order to justify the use of Fubini's theorem and the dominated convergence theorem. We will also use the four-point function bound in Section \ref{sec::gaussianity} to prove the estimate described in Proposition \ref{lemma::voronoi_estimates}, which is essential to showing Gaussianity of the process.

\begin{prop}\label{prop:K_a.s._finite}
Fix D and let $z_1,\cdots, z_4 \in D$. Then there exists some universal constant $C>0$ such that for $\eps$ with $B_{z_i}(\eps)\subset D$ for all $i$:
	\begin{equation}
	\label{eqn::mollifiedfield_2pointbound}
	\E\left[\prod_{i=1,2} \tilde{h}^D_\eps(z_i)\right]  \leq C \big(l_2(z_1,z_2)\big)^{1/2}
	\;\;\text{ and } \;\;
	\E\left[\prod_{1\leq i \leq 4} \tilde{h}^D_\eps(z_i)\right] \leq C \big(l_4(z_1,z_2,z_3,z_4)\big)^{1/4}.\end{equation} In particular, using Definition \ref{remark::K_alt_def}, we see that $$\tilde K_2^D(z,w) ^2 \le Cl_2(z_1,z_2) \;\;\text{ and } \;\;
	\tilde K_4^D(z_1,\ldots,z_4 )^4 \le Cl_4(z_1,z_2,z_3,z_4)
	.$$
\end{prop}
\begin{remark}\label{remark::K_4_neater}
Using the fact that $R(z_i;z_1,\cdots, z_4)/R(z_i,D)\leq 1/10$ for all $i\leq 4$, the AM-GM inequality and Koebe's quarter theorem, we see that we can also write $$|\tilde{K}_4^D(z_1,\cdots, z_4)| \leq C \sum_{i\ne j} \left(\log^2\left(\frac{|z_i-z_j|}{4\,\text{diam} (D)}\right)\vee \log^2(10)\right).$$
This alternative formulation will be useful in Section \ref{sec::gaussianity}.
\end{remark}

We first prove an intermediate lemma.
Let $\phi, (\phi_r^z)_{r>0,z\in D}: \C \to \R$ be as in \cref{remark::K_alt_def} and $(\tilde{h}^D_r(z))_{r>0,z\in D}$ be the mollified field. Then we have the following:

\begin{lemma}\label{Variance} 
	Fix $D \subset \C$. There exists $C>0$ universal such that for all $z,r$ with $r\leq R(z,D)/10$,
	$$
	\var (\tilde{h}^D_r(z)) \le C\log (R(z,D)/r).
	$$
	Also,
	$$
	\E[  (\tilde{h}^D_r(z))^4 ] \le C (\log^2(R(z,D)/r)   +  \log(R(z,D)/r))
	$$
\end{lemma}
\begin{proof}
	Let $N =\lfloor \log_2 (R(z,D)/5r) \rfloor$ and set $B_k = B_z(2^kr)$ for $k \le N$; $B_{N+1} = D$.
	By the domain Markov property, we can write
	\begin{equation}
	h^D = h_D^{B_N} +\ph_D^{B_N} \label{eq:DMP_break_up}
	\end{equation}
	where $\ph_D^{B_N}$ is harmonic in $B_N$ and $h^{B_N}_D$ is independent of $\ph_D^{B_N}$ and is 0 outside $B_N$.
	\begin{figure}[h]
		\centering
		\includegraphics[scale=0.5]{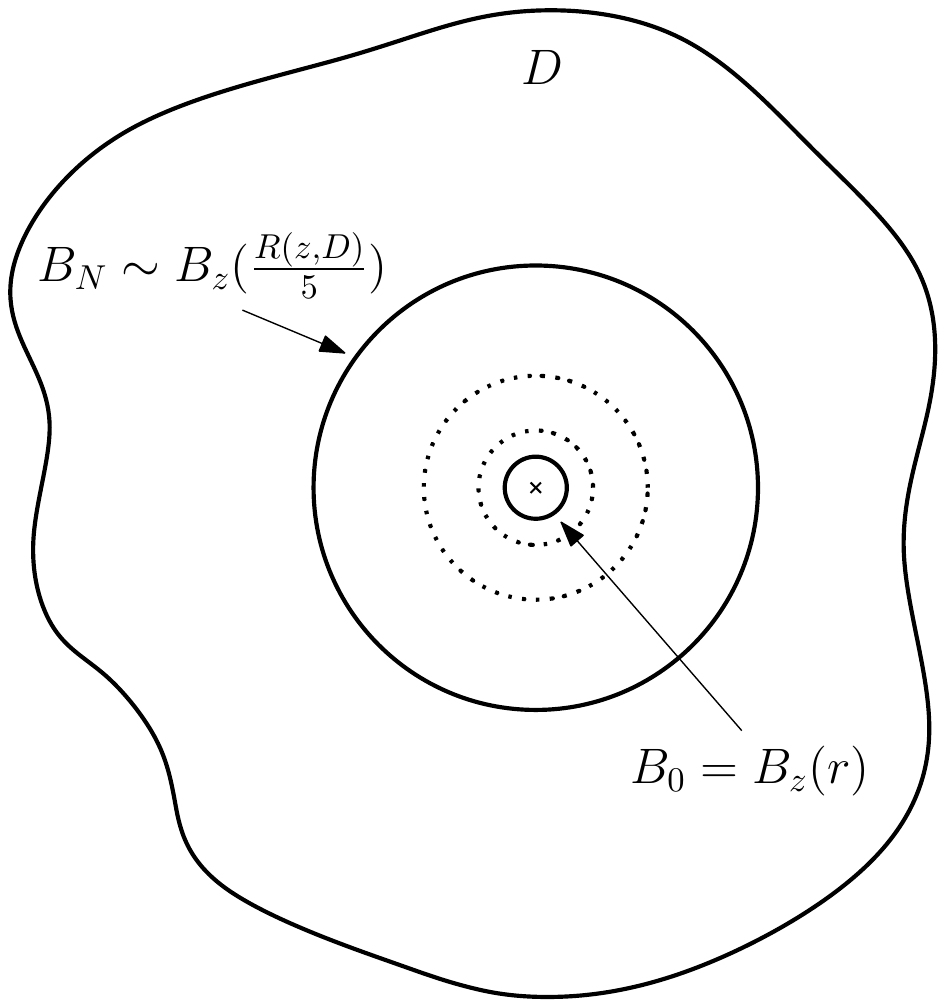}
		\caption{The sets $(B_k)_{0\leq k \leq N}$ (the dotted circles represent the boundaries of $(B_k)_{1\leq k \leq N-1}$) and $B_{N+1}=D.$}
	\end{figure}
	Iterating this decomposition, we get
	$$
	h^D = \tilde{h}+ \sum_{k=0}^N \ph_k
	$$
	where:
	\begin{itemize}
		\item the $\ph_k$'s are independent and $\ph_k$ is harmonic in $B_k$;
		\item $\tilde{h}$ is an independent copy of $h^{B_0}$ and is $0$ outside of $B_0=B_z(r)$.
	\end{itemize}
	Recall that $\phi^z_r$ is radially symmetric (about $r$) and has mass $1$, so that
	\begin{equation}(\ph_k , \phi^z_r) = \ph_k(z)\label{eq:break_up}
	\end{equation}
	for every $0\le k \le N$. Note that by scale and translation invariance we have $(\tilde{h},\phi_z^r)\overset{(d)}{=}(h^{\mathbb{D}},\phi)$, and therefore $(\tilde{h},\phi^z_r)$ has finite variance (by Assumptions \ref{ass:ci_dmp}) that is independent of $r$ and $z$. Also note that since $\ph_k$ is equal (in law) to the harmonic part in the decomposition
	$h^{B_{k+1}}=h_{B_{k+1}}^{B_k}+\ph_{B_{k+1}}^{B_k}$, we have by conformal invariance and the domain Markov property that $$\{\ph_k(2^{k+1}rw+z): w\in B_0(1/2)\}\overset{(d)}{=}\{\ph^{B_0(1/2)}_\mathbb{D}(w): w\in B_0(1/2)\}$$ for $0\le k \leq N-1$.
	Combining this information and \eqref{eq:break_up}, we finally obtain that
	\begin{equation*}
	\var(\tilde{h}^D_r(z))=\var(h^D,\phi_r^z)  = \var(h^\D,\phi) + N \var(\ph^{B_0(1/2)}_{\mathbb{D}}(0))+\var (\ph^{B_N}_D(z)).
	\end{equation*}
	This completes the proof using our finite variance assumption. Note that $\text{Var}(\ph^{B_N}_D(z))$ can be bounded above by something which does not depend on either $z$ or $D$. Indeed, by the Koebe quarter theorem, we can conformally map $D$ to $\mathbb{D}$, with $z\mapsto 0$ and $B_N \mapsto D_N$, for some $D_N\subset \D$ such that $d(0,\partial D_N)\geq 1/40$. Then by conformal invariance and Lemma \ref{lemma::harm_av_monotone},  $\var(\ph^{B_N}_D(z))=\var(\ph^{D_N}_\mathbb{D}(0))\leq \var(\ph^{B_0(1/40)}_{\mathbb{D}}(0))$.
	\medbreak
	
	Using the same decomposition, \eqref{eq:DMP_break_up} and \eqref{eq:break_up}, and the fact that every variable in the decomposition has mean 0, we also obtain the fourth moment bound
	\begin{equation*}
	\E[\tilde h^D_r(z)^4] \le C' (N^2+N)
	\end{equation*}
	for some constant $C' >0$.
\end{proof}

We now prove a corollary which gives the same bound for the variance of the field convolved with a mollifier at a point that is near the boundary.
\begin{corollary}\label{cor:var_near_boundary}
	There exists a constant $c>0$ such that for any point $z$ with $R(z,D)/10 < r<d(z,\partial D)$,
	$$
	\var(\tilde h^D_r(z)) \le c.
	$$
\end{corollary}
\begin{proof}
	We can find a domain $D' $ containing $D$ such that $10r \le R(z,D') \le 11r $. Also we can write $$h^{D'} = h_D^{D'} + \ph_D^{D'}$$ for $h_D^{D'}\overset{(d)}{=}h^D$ and $\ph_D^{D'}$ independent and harmonic inside $D$. We know from \cref{Variance} that $\var(h^{D'} ,\phi^z_r) \le c$. Since adding $\ph_D^{D'}$ only increases the variance, the proof is complete.
\end{proof}

We now extend this to the full covariance structure of the mollified field to prove Proposition \ref{prop:K_a.s._finite}.

\begin{proof}[Proof of Proposition \ref{prop:K_a.s._finite}] We first prove \eqref{eqn::mollifiedfield_2pointbound} in the case of two points $z_1\ne z_2$. Observe that by the domain Markov property, as in the proof of Proposition \ref{lemma::conv_circ_avg_field}, if $\ve_0:= |z_1-z_2|/10 \wedge R(z_1,D)/10 \wedge R(z_2,D)/10$ then for all $\eps<\eps_0$ we have that $\E[\tilde{h}_\eps^D(z_1)\tilde{h}_\eps^D(z_2)]=\E[\tilde{h}_{\eps_0}^D(z_1)\tilde{h}_{\eps_0}^D(z_2)]$. Thus we need only prove the inequality for $\eps_0\le \eps< d(z,\partial D)$. However, this follows simply by applying Cauchy--Schwarz and using Lemma \ref{Variance} and/or Corollary \ref{cor:var_near_boundary} as necessary (depending on whether $\eps_0$ is less than or greater than $R(z_1,D)/10$ and $R(z_2,D)/10$). The case of four points follows in the same manner.
\end{proof}

\subsection{Identifying the two point function} 
\label{sec:twopointGreen}
In this section we prove that for $z_1,z_2$ distinct
\begin{equation}
\label{eqn::lim_circav_kernel}
\tilde{K}_2^D(z_1,z_2):=\lim_{\eps\to 0} \E[h^D_\eps(z_1)h^D_\eps(z_2)]= aG^D(z_1,z_2)\end{equation}
for some $a>0$, where $G^D$ is the Green's function on $D$ with Dirichlet boundary conditions.

We first need a technical lemma, namely, an exact expression for the variance of harmonic averages, derived from the bounds of the previous section together with the properties of the two-point kernel deduced in Section \ref{sec:harmonic}.

\begin{lemma}\label{lemma::variance_harm_avg}
	Let $\gamma$ be the boundary of a Jordan domain $D'\subset D$,
	such that $\gamma \cap \partial D=\emptyset.$
	Let $z\in D'$. Then
	$$\E[(h^D,\rho_z^\gamma)^2] = \int_D \tilde{K}_2^D(w,z) \rho_z^\gamma(dw) $$
	where $\rho_z^\gamma$ is the harmonic measure seen from $z$ on $\gamma$.
\end{lemma}

Note that although the statement of this lemma may seem obvious, recall from Section \ref{sec::harm_avg} that the notation for the harmonic average $(h^D,\rho_z^\gamma)$ is an abuse of notation (the way we define it does not a priori have anything to do with integrating against harmonic measure).

\begin{proof}
	Let $\varphi:D'\to \D$ be the unique conformal map with $\varphi(z)=0$ and $\varphi'(z)>0$. Then by definition of the harmonic average,
	\begin{eqnarray*}
		\E[(h^D,\rho_z^\gamma)^2] & = & \lim_{\delta_2 \to 0}\lim_{\delta_1 \to 0 } \E[(h^D, \hat{\psi}^{\delta_1}_z)(h^D,\hat \psi_z^{\delta_2})] = \lim_{\delta_2 \to 0}\lim_{\delta_1 \to 0 }\iint_{D^2} \tilde{K}_2^D(x,y)\hat \psi_z^{\delta_1}(x)\hat \psi_z^{\delta_2}(y) \, dx dy \\
		& = &\lim_{\delta_2 \to 0}\lim_{\delta_1 \to 0 }\iint_{\D^2} \tilde{K}^D_2(\varphi^{-1}(x),\varphi^{-1}(y)) \psi_0^{\delta_1}(x)\psi_0^{\delta_2}(y)\, dx dy, \end{eqnarray*}
	where the last equality follows by definition of $\hat{\psi}^\delta_z$ and the harmonic average. Recall that $\psi_0^{\delta}$ is defined by normalising a smooth radially symmetric function from $\D$ to $[0,1]$, that is equal to 1 on $\{z: 1-\delta\leq |z| \leq 1-\delta/2\}$ and $0$ on the $\delta/10$ neighbourhood of this annulus, to have total mass $1$.
	
	We define $$\tilde{K}_2^D(\varphi^{-1}(x),\varphi^{-1}(y))=: f(x,y).$$
	Observe that for every $x\in \D$, by analyticity of $\varphi$ and Proposition \ref{lemma::harmonic kernel}, $f(x,y)$ viewed as a function of $y$ is harmonic in $\D\setminus \{x\}.$ We also have the bound
	\begin{equation*}  f(x,y)\leq C \log|\varphi^{-1}(x)-\varphi^{-1}(y)|\end{equation*}
	for every $x\ne y$ and some $C=C(D)$ by Proposition \ref{prop:K_a.s._finite}. The dependence on the domain here comes from the bounded conformal radius term in (\ref{eq:rzw}). Now fix $\delta_2>0$ and take $\delta_1<\frac{4}{11} \delta_2$, so that the support of $\psi_0^{\delta_1}$ lies entirely outside of $B_0(1-4{\delta_2}/10)\supset \text{supp}(\psi_0^{\delta_2})$. Pick $x \in \text{supp}(\psi_0^{\delta_1})$. Then it follows from harmonicity of $f(x,y)$ in
	$B_0(1-4{\delta_2}/10)$ that
	\begin{equation}
	\label{eqn::harmonic_rewriting} \int_\D f(x,y)\psi_0^{\delta_2}(y)\, dy = f(x,0).
	\end{equation}

	Now, (\ref{eqn::harmonic_rewriting}) tells us that (since the above expression does not depend on $\delta_2$)
	$$\E[(h^D,\rho_z^\gamma)^2]=\lim_{\delta_1 \to 0 } \int_\D f(x,0) \psi_0^{\delta_1}(x) \, dx.$$
	Furthermore, Proposition \ref{lemma::harmonic kernel} together with the fact that $\gamma$ lies strictly within $D$, implies that $f(x,0)$ extends to a continuous function on $x\in \partial \D$. This means that the right hand side is equal to $\int_{\D} f(x,0) \, \rho_0^{\partial \D}(dx)$, which is equal to $\int_D \tilde{K}_2^D(w,z)\rho_z^\gamma(dw)$ by a change of variables.
\end{proof}

\begin{remark}\label{remark::circ_avg_int_form}
	As a direct consequence of the above proof we see that if $\eps<1$ then
	$$\mathbb{E}[(h^\D_{\eps}(0))^2] = \int_{\D} \tilde{K}_2^{\D} (0,y) \rho_0^{\partial B_0(\eps)}(dy). $$
\end{remark}

We are now ready to prove \eqref{eqn::lim_circav_kernel}: we start with the case $x=0$ and $D=\D$.

\begin{lemma}\label{lem:k2_unit_disc} There exists $a>0$ such that
	$\tilde{K}_2^{\D}(0,y)=-a\log|y|$ for all $y\in \D\setminus \{0\}$.
\end{lemma}

\begin{proof}

First, we prove that there exists an $a>0$ such that $f(r):= \E[h_r^\D(0)^2]$ is equal to $-a\log(r)$ for all $r\in [0,1]$. To see this, note that by the domain Markov property and conformal invariance we have $f(rs)=f(r)+f(s)$ for all $r,s<1$. Moreover, $f$ is continuous (by Remark \ref{remark::circ_avg_int_form} and Lemma \ref{lemma::harmonic kernel}) and decreasing (by Lemma \ref{lemma::harm_av_monotone}), with $f(1)=0$. This proves the claim.

With this in hand, by Remark \ref{remark::circ_avg_int_form} we can write
	$$-a\log |y| = \E[(h_{|y|}^\D(0))^2]= \int_{\D } \tilde{K}_2^\D(0, w) \rho_0^{\partial B_{0}(|y|)}(dw),$$ where by conformal invariance (in particular, rotational invariance) $\tilde{K}_2^\D(0,w)$ must be constant and equal to $\tilde{K}_2^\D(0,|y|)$ on $\partial B_{0}(|y|)$. Since $\rho_0^{\partial B_{0}(|y|)}(\cdot)$ has total mass $1$ we obtain the result.
 \end{proof}

In particular, combining this with conformal invariance (Proposition \ref{lemma:conformal_invariance}) and Lemma \ref{lemma::circ_avg_good_approx}, we obtain:

\begin{corollary}\label{cor::k_green}
$\tilde{K}_2^D=aG^D$, where $G^D$ is the Green's function with zero boundary conditions and $a\ge 0$ is some constant.
\end{corollary}

\subsection{The circle average approximates the field}

We conclude this section by showing that, in fact, the covariance kernel $K_2^D$ defined in Assumptions \ref{ass:ci_dmp} (which we recall is a bilinear form on $C_c^\infty(D)\times C_c^\infty(D)$) corresponds to integrating against the two-point function $\tilde{K}_2^D$. Thus, due to Corollary \ref{cor::k_green}, we can say that our field has ``covariance given by a multiple of the Green's function''. In particular, there exists $a>0$ such that for any test function $\phi \in C^\infty_c(D)$,
	\begin{equation}\label{eqn:varhDphi}
	\var (h^D,\phi) =  \iint_{D^2}  aG^D (x,y) \phi(x) \phi(y) dx dy.
	\end{equation}

\begin{lemma}\label{lemma::circ_avg_good_approx}
	For any $\psi_1,\psi_2 \in C_c^\infty(D)$
	$$ K^D_2(\psi_1,\psi_2)=\iint_{D^2} \tilde{K}_2^D(x,y) \psi_1(x)\psi_2(y) \, dx dy.$$
	In particular, if $h_\eps^D$ is the circle average field and $\psi\in C_c^\infty(D)$, then $$\Var ((h^D_\eps, \psi)-(h^D,\psi)) \to 0$$ as $\eps \to 0$. 	
\end{lemma}

We will need this last statement for the conclusion of the proof: see Section \ref{sec:Concl}.
\begin{proof}
	We have, by Proposition \ref{prop:K_a.s._finite} and dominated convergence (for this we use that $\psi_1$ and $\psi_2$ are compactly supported, meaning that for some $\eps_0>0$, $B_x(\eps_0)\subset D$ for all $x\in \text{Support}(\psi_1)\cap \text{Support}(\psi_2)$),
	\begin{eqnarray*} \iint_{D^2} \tilde K_2^D(x,y)\psi_1(x)\psi_2(y) \, dx dy & = & \iint_{D^2} \lim_{\eps\to 0 } \E[\tilde{h}_\eps^D(x)\tilde{h}_\eps^D(y)] \psi_1(x) \psi_2(y) \, dx dy \\
		& = & \lim_{\eps \to 0 }\iint_{D^2}  \E[\tilde{h}_\eps^D(x)\tilde{h}_\eps^D(y)] \psi_1(x) \psi_2(y) \, dx dy \\
		& = & \lim_{\eps\to 0} \iint_{D^2} K_2^D( \psi_1(x)\phi_{\eps}^x,  \psi_2(y)\phi_{\eps}^y) \, dx dy ,
 \end{eqnarray*}
	where the last line follows from definition of $K_2^D$. Here $K_2^D( \psi_1(x)\phi_{\eps}^x,\psi_2(y)\phi_\eps^y )$ means the value of $K_2^D(f,g)$ where $f:z\mapsto \phi_{\eps}^x(z)\psi_1(x)$ and $g: z\mapsto \phi_{\eps}^y(z)\psi_1(y)$ are both in $ C_c^\infty(D)$. Now we use the fact that $K_2^D$ is a continuous bilinear form on $C_c^\infty(D)\times C_c^\infty(D)$ with the topology discussed in the introduction. This means that if we fix $y\in D$ and consider the map $f\mapsto K_2^D(f,\phi_\eps^y\psi_2(y) )$, then this is a continuous linear map on $C_c^\infty(D)$ i.e. it is a distribution. Standard theory of distributions (associativity of convolution, see for example \cite[Theorem 6.30]{rudin}), then tells us that
	\begin{equation}\label{eqn::conv_dist}
	\int_D K_2^D(\phi_\eps^x \psi_1(x), \phi^y_\eps\psi_2(y)) \, dx = K_2^D(\psi_1 * \phi_\eps, \phi^y_\eps\psi_2(y))
	\end{equation}
	where $\psi_1 * \phi_{\eps}(z)=\int_D \psi_1 (x) \phi_{\eps}^z (x) \, dx$. Now applying the same argument in the $y$-variable gives that the right hand side of (\ref{eqn::conv_dist}) is equal to $K_2^D(\psi_1 * \phi_{\eps}, \psi_2* \phi_{\eps})$, and we have overall attained the equality
	\begin{equation}\label{eqn::conv_dist_2}
	\iint_D \tilde K_2^D(x,y)\psi_1(x)\psi_2(y) \, dx dy = \lim_{\eps\to 0} K_2^D(\psi_1 * \phi_{\eps}, \psi_2* \phi_{\eps}).
	\end{equation}
	Finally, since $\psi_i*\phi_{\eps}\to \psi_i$ in $C_c^\infty(D)$ for each $i$ as $\eps\to 0$ \cite[\S5.3, Theorem 1]{evans}, and $K_2^D$ is continuous, we can deduce the result.
	
	For the statement concerning the variance, we expand
	\begin{align*} \Var ((h^D_\eps, \psi)-(h^D,\psi)) & = \E[(h^D_\eps,\psi)^2]+\E[(h^D,\psi)^2]-2\E[(h^D_\eps,\psi)(h^D,\psi)] \\
	& = \iint K_2^D( \psi(x)\phi_{\eps}^x,  \psi(y)\phi_{\eps}^y) \, dx dy+ K_2^D(\psi,\psi)-2\int K_2^D(\psi(x)\phi_\eps^x,\psi) \, dx \\
	& = K_2^D(\psi * \phi_{\eps},\psi * \phi_{\eps})+K_2^D(\psi,\psi)-2K_2^D(\psi*\phi_{\eps},\psi)
	\end{align*} where the final equality follows by the same reasoning that led us to \eqref{eqn::conv_dist_2}. Again, since $\psi * \phi_\eps \to \psi$ in $C_c^\infty(D)$ as $\eps\to 0$, this allows us to conclude that the final expression converges to $0$ as $\eps\to 0$.
\end{proof}

Similarly, we deduce the following:

\begin{lemma}\label{lemma::4_point_kernel}
	For any $\psi_1,\cdots, \psi_4 \in C_c^\infty(D)$
	$$ K_4^D(\psi_1,\psi_2,\psi_3,\psi_4)=\iint_{D^4} \tilde{K}_4^D(x_1,x_2,x_3,x_4) \prod_{1\le i \le 4} \psi_i(x_i) \, dx_i.$$ 	
\end{lemma}

\begin{remark}
	Lemma \ref{lemma::circ_avg_good_approx} and Corollary \ref{cor::k_green} imply that Assumptions \ref{ass:ci_dmp} (ii) (Dirichlet boundary conditions) is satisfied by a much wider family of test functions $f_n$: in particular the assumption that $f_n$ be rotationally symmetric in this assumption can be partly relaxed (however $f_n$ cannot be completely arbitrary, i.e., it is not sufficient to assume that the support of $f_n$ leaves any compact and that $f_n$ has bounded mass, as can be seen by considering $f_n$ to have unit mass within a ball of radius $1/n$ at distance $1/n$ from the boundary).
\end{remark}

%
%

\section{Gaussianity of the circle average}\label{sec::gaussianity}

In this section, we argue that from Assumptions \ref{ass:ci_dmp}, we can deduce that the circle average field of $h^D$ is Gaussian. This is where we will need to use our finite fourth moment assumption. Let $(h_\eps^D(z))_{z\in D}$ be the circle average field. The key result we prove here is the following:

\begin{prop}\label{prop:circle_avg_pointwise_Gaussian}
Let $z_1,\ldots,z_k$ be $k$ pairwise distinct points in $D$ with $d(z_i,z_j)>2\eps$ for every $1\leq i \ne j\leq k$ and $d(z_i,\partial D)>2\eps$ for $1\le i \le k$. Then the law of $(h^D_\ve(z_1),\ldots,h^D_\ve(z_k))$ is that of a multivariate Gaussian random variable.
\end{prop}

\subsection{Bounds for the 4 point kernel}
\label{sec:expression_kernel}

Let $z_1,\cdots, z_4$ be pairwise distinct points in $D=\D$ and let $$V_i = \{y\in \D: |u(y)-u(z_i)|_1<|u(y)-u(z_j)|_1 \; \forall j\ne i\}$$ for $1\leq i\leq 4$, where $u(x)=x/|x|$ and $|\cdot |_1$ is distance (with respect to arc length) on the unit circle. In words, we divide the disc into four wedges each containing one of the four distinguished points. By definition, the boundary between two adjacent wedges $V_i$ and $V_j$ is the ray emanating from the origin which bissects the rays going through $z_i$ and $z_j$.

\begin{figure}[h]
	\centering
	\includegraphics[scale=0.5]{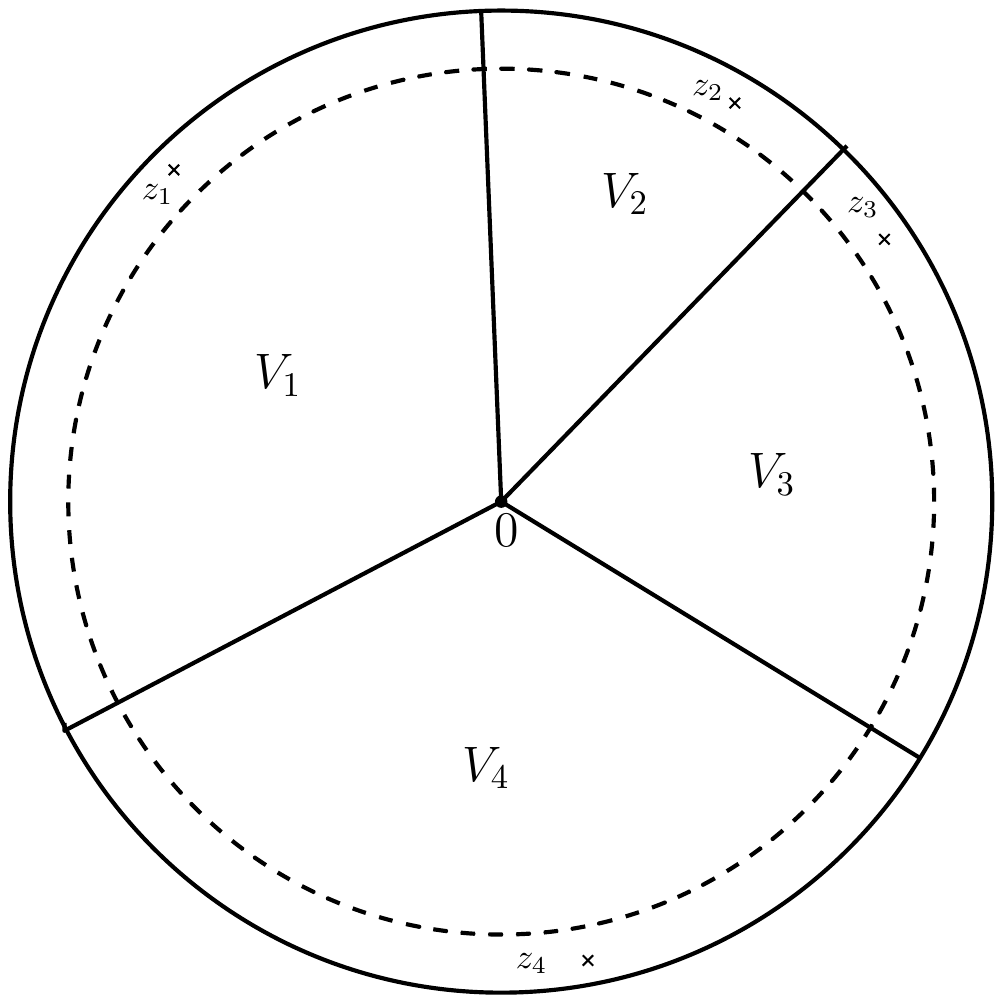}
	\caption{The cells $V_1,\cdots, V_4$}
\end{figure}

We have (by definition of the harmonic average, and Lemma \ref{lemma::fourpointexpression}) the following expression for the four point kernel:

	\begin{equation*}
 \tilde K_4^\D(z_1,z_2,z_3,z_4) = \E\left[\prod_{i=1}^4(h^\D,\rho^{\partial V_i}_{z_i}) \right]
	\end{equation*}

In the next section, we will require some bounds on these quantities when the $z_i$'s are close to the boundary of $\D$. We can estimate them as follows:

\begin{prop}\label{lemma::voronoi_estimates}
Suppose that $z_1,\cdots, z_4$ are pairwise distinct points in $\C$, each with modulus between $1-\eps$ and $1$. Then if $V_j$ is as described above (with respect to $z_1,\cdots, z_4$) and $a_j=\min_{i\ne j} \{|u(z_j)-u(z_i)|_1\}$ is the isolation radius of $u(z_j)$ in $\{u(z_1),\cdots, u(z_4)\}$ we have
$$ \mathbb{E}[(h^\D,\rho_{z_j}^{\partial V_j})^4]\leq c(\frac{\eps^4}{a_j^4}\wedge 1)\log^4(a_j)$$ for some universal constant $c$.
\end{prop}
We remark that the bound above is much improved compared to \cref{prop:K_a.s._finite} if $\ve \ll \min_ja_j$. This is where the effect of the Dirichlet boundary condition assumption is manifested. Also the choice of the Voronoi cell is not crucial, any partition of the domain separating the points would work. This particular choice of cells is simply to make the calculations explicit.
\begin{proof}
First suppose that $a_j>\eps$. By Lemma \ref{lemma::harm_av_monotone} and Cauchy--Schwarz, it is enough to consider the wedge $$W_a=\{z=r\e^{i\theta}: -a<\theta<a\,,\, 0<r<1\}$$ for every $\eps\leq a \leq \pi/2$ and prove that $\E[(h^\D,\rho_w^{\partial W_a})^4]\leq c\frac{\eps^4}{a^4}\log^4(a)$ when $w:=1-\eps$. To begin, we describe how to approximate $(h^\D,\rho_{w}^{\partial W_a})$ in a slightly different way. This is very similar to the approximation used in Section \ref{sec::harm_avg} (we take some smooth approximations to the harmonic measure on the boundary of a sequence of domains increasing to $W_a$ from the inside) but is more explicit, which will be an advantage here.

For $\delta\ll \eps$, let $$r_1^\delta:=\{r\e^{i(a-\delta)}: \delta \leq r\leq 1-\delta\}\;\; \text{and} \;\; r_2^\delta := \{r\e^{i(-a+\delta)}: \delta \leq r\leq 1-\delta\}$$ and let $W_a^\delta = \{ \delta < |z| <1 - \delta; \arg(z) \in (- a + \delta, a - \delta)\}$.
Let $\hat \nu^\delta$ be the harmonic measure seen from $w$ on the boundary of the domain $W_a^\delta$ and let $\nu^\delta$ be the same harmonic measure, but restricted to the lines $r_1^\delta$ and $ r_2^\delta$. Finally, let $\phi$ be a smooth radially symmetric function with mass $1$, supported on $\D$, and denote $\phi^z_\delta(\cdot)=\delta^{-2}\phi(|z-\cdot|/\delta)$ as usual. Set $p^\delta(z) = \int \phi^z_{\delta/10}(x) \nu^\delta (dx)$. We claim the following.
\medskip

\noindent\textbf{Claim:}
\begin{enumerate}[(a)]
\item $(h^\D,p^\delta)\to (h^\D, \rho_{w}^{\partial W_a})$ in $L^2(\P)$ and in probability as $\delta \to 0$.
\item $p^\delta(z)$ is bounded above by some universal constant times $\delta^{-1}\frac{\eps}{a}\frac{\pi}{2(a-\delta)}(\frac{|z|}{1-\delta})^{\frac{\pi}{2(a-\delta)}-1}$.
\end{enumerate}
\begin{proof}[Proof of claim]
 For (a), we first prove the same statement with $p^\delta$ replaced by $$\hat{p}^\delta(z):=\int_\D \phi^z_{\delta/10}(x)\hat{\nu}^\delta(dx).$$
 To do this, we apply the Markov property in $W_a$, writing $h^\D=h_\D^{W_a}+\varphi_{\D}^{W_a}$.
 First, we consider the part with zero boundary conditions: by \cref{cor::k_green}, we have that
 $$
  \E[(h_\D^{W_a}, \hat p^\delta)^2]  = \int_{W_a^2}G^{W_a}(z,w) \hat p^\delta (z)\hat p^\delta(w) dz dw \to 0
  $$ as $\delta \to 0$ by standard properties of the Dirichlet Green's function (note that $\hat p^\delta$ is simply a perturbation of the harmonic measure on $\partial W_a$).

  Then we consider the harmonic part: we have
   $$(\varphi_{\D}^{W_a},\hat{p}^\delta)=(\varphi_{\D}^{W_a},\phi_{\delta/10}^0 * \hat \nu^\delta)=(\varphi_{\D}^{W_a}*\phi_{\delta/10}^0,\hat \nu^\delta)=(\varphi_{\D}^{W_a},\hat \nu^\delta)=\varphi_{\D}^{W_a}(w)$$ for every $\delta$, since $\phi$ is radially symmetric with mass $1$, $\varphi_{\D}^{W_a}$ is harmonic, and $\hat{\nu}$ is the harmonic measure on $W_a^\delta\subset W_a$ meaning that $\varphi_\D^{W_a}$ is a true harmonic function on $W_a^\delta$. Combining these two facts, it follows  that
  $(h^\D, \hat{p}^\delta)$ converges to $(h^\D, \rho_{w}^{\partial W_a})$ in $L^2(\P)$ and in probability as $\delta \to 0$. Now to conclude (a), simply observe that $\var (h^\D,p^\delta-\hat{p}^\delta)$ converges to $0$ as $\delta \to 0$: again, this follows from
 Corollary \ref{cor::k_green} and elementary properties of the Green's function since $p^\delta  - \hat p^\delta$ is supported on an arbitrarily small neighbourhood of a fixed arc of the unit circle (and converges to the harmonic measure on that arc seen from $w$).


 We now move on to (b). For $z\in \D$ we have $$p^\delta(z)\leq \sup_x |\phi(x)|\times \delta^{-2} \,\nu^\delta(B_z(\delta)\cap\{r_1^\delta\cup r_2^\delta\})= \sup_x |\phi(x)| \times \delta^{-2} \,\hat{\nu}^\delta(B_z(\delta)\cap\{r_1^\delta\cup r_2^\delta\})$$ by definition.
 Consider the maps
 $$\varphi_1^\delta: z\mapsto \frac{z^{\pi/(2(a-\delta))}}{(1-\delta)^{\pi/(2(a-\delta))}}\;, \;\; \varphi_2^\delta: z\mapsto \frac{z^2+2z-1}{2z+1-z^2} \;,\;\; \varphi_3^\delta: z\mapsto \frac{z-(1-\eta)}{1-(1-\eta)z}$$ where  $(1-\eta)=\varphi_2^\delta \circ \varphi_1^\delta (w).$
 Then $\ph_1^\delta$ maps $\tilde{W}_a^\delta$ to the half disc $ \D \cap \{\Re (z) >0 \}$, $\tilde{W}_a^\delta = \{ |z| < 1- \delta: \arg (z) \in ( - a + \delta, a - \delta)\}$. It can also be checked using elementary properties of M\"obius maps that $\ph_2^\delta$ maps the half disc to the full disc $\D$, and $\ph_3^\delta$ maps $\D$ to itself so that $\ph_2^\delta \circ \ph_1^\delta(w)$ is sent to $0$. Hence $\varphi_3^\delta \circ \varphi_2^\delta \circ \varphi_1^\delta$ is a conformal map from $\tilde{W}_a^\delta$ to $\D$ sending $w$ to $0$,

A computation verifies that for any $z\in \D$, $\varphi_3^\delta \circ \varphi_2^\delta \circ \varphi_1^\delta(B_z(\delta)\cap\{r_1^\delta\cup r_2^\delta\} )$ is an arc of the unit circle with length less than $\frac{c\eps \delta \pi}{2a(a-\delta)}(\frac{|z|}{1-\delta})^{\frac{\pi}{2(a-\delta)}-1}$ for some universal constant $c$. In particular, we use that \begin{equation*}
 |(\varphi_1^\delta)'|\leq \frac{\pi}{2(a-\delta)}\left(\frac{|z|}{1-\delta}\right)^{\frac{\pi}{2(a-\delta)}-1}\;\;\text{on} \;\; \{r_1^\delta \cup r_2^\delta \}; \;\;\;  |(\varphi_2^\delta)'|\leq 4 \;\; \text{on} \;\; \{iy:y\in [-1,1]\} \;\; \end{equation*}
 and
\begin{equation*}|(\varphi_3^\delta)'|\leq 2\eta \;\; \text{on} \;\; \{\e^{i\theta}: \pi/2 \leq \theta \leq 3\pi/2\} \end{equation*}
where $\eta\leq \frac{1}{a}c\eps$ for some such $c$. By definition of $\hat{\nu}^\delta$, and the fact that the harmonic measure with respect to $W_a^\delta$ is less than the harmonic measure with respect to $\tilde{W}_a^\delta$ for any fixed subset of $\{r_1^\delta \cup r_2^\delta \}$, this finishes the proof of the claim.
\end{proof}

With this claim in hand, we have by Fatou's Lemma and Lemma \ref{lemma::4_point_kernel}
$$ \mathbb{E}[(h^\D,\rho_{w}^{\partial W_a})^4] \leq \liminf_{\delta\to 0} \iint_{\D^4} \tilde{K}_4^\D(x_1,\cdots,x_4) p^{\delta}(x_1) \cdots p^{\delta}(x_4) \, \prod dx_i $$
and then by Proposition \ref{prop:K_a.s._finite} and Remark \ref{remark::K_4_neater}, we see that this is less than or equal to
$$c \frac{\eps^4}{a^4} \liminf_{\delta\to 0} \delta^{-4} \iint_{\{\text{supp}(p^\delta)\}^4 } \big( 1+ \sum_{i\ne j} \log^2(|x_i-x_j|)\big) \prod_{i=1}^4 \frac{\pi}{2(a-\delta)}\left(\frac{|x_i|}{1-\delta}\right)^{\frac{\pi}{2(a-\delta)}-1} dx_i
$$
another universal $c$ (which may now change from line to line).

We can simplify this expression.
Because $p^\delta$ is supported in a strip of width $\delta/10$ around the lines $r_1^\delta$ and $r_2^\delta$, we can change of variables by considering the orthogonal projection onto $r_1^\delta \cup r_2^\delta$, so that we can write
$$(x_i)_{1\le i \le 4} = (z_i + y_i)_{1\le i \le 4} \text{ with }z_i \in r_1^\delta \cup r_2^\delta; \  0 \le |z_i| \le 1-\delta \text{ and }-\delta/10 \le |y_i| \le \delta /10.$$
Note then that $ \log^2 | x_i - x_j| \le \log^2 ( \big||z_i| - |z_j| \big|) $. Performing the change of variables $(u_i)_{1\leq i\leq 4} = (\frac{|z_i|}{1-\delta})_{1\leq i \leq 4}$) we obtain that the above is less than or equal to
$$c\frac{\eps^4}{a^4} \left(1+\liminf_{\delta\to 0} \sum_{1\leq i\ne j\leq 4} \int_{[0,1]^4} \log^2|u_i-u_j| \prod_{i=1}^4 \frac{\pi}{2(a-\delta)} u_i^{(\frac{\pi}{2(a-\delta)}-1)} du_i\right).$$
Thus, to conclude the proof in the case $a_j\ge \eps$, we need to show that
$$ \int_{[0,1]^2} \log^2|x-y| bx^{b-1}by^{b-1} \, dx \, dy \leq C\log^4 b $$
for some constant $C$ and all $b\geq 1$. To see this, we break up the integral into 4 regions. The first is $S_1:=\{x\leq 1-\log(b)/b\}\cap \{y\leq 1-\log(b)/b\}$, and on
this region $bx^{b-1}$ and $by^{b-1}$ are uniformly bounded in $b$ (indeed, one can easily check that $b(1-\log b/b)^{b-1}\to 1$ as $b \to \infty$). Since $\iint_{[0,1]^2} \log^2|x-y| dx dy$ is finite, this means that the integral over $S_1$ is less than or equal to a universal constant. The second is $S_2:=\{x \leq 1-\log(b)/b\}\cap\{y> 1-\log(b)/b\}$, and on this region, $bx^{b-1}y^{b-1}$ is uniformly bounded in $b$ for the same reason. Thus integrating over $S_2$, and using that $\int_0^a \log^2(u) \, du = O(a\log^2(a))$ as $a\to 0$, we obtain something of order at most $\log^3(b)$. Symmetrically, the integral  over the region $S_3:=\{y \leq 1-\log(b)/b\}\cap\{x> 1-\log(b)/b\}$ is at most order $\log^3(b)$. The last region is $S_4:=\{x\geq 1-\log(b)/b\}\cap \{y\geq 1-\log(b)/b\}$. Using that $\iint_{[0,a]^2} \log^2(|x-y|) dx dy=O(a^2\log^2(a))$ as $a\to 0$, we see that the integral over $S_4$ is $O(\log^4(b))$, and this completes this part of the proof.

Finally, suppose that $a_j<\eps$. Then we have $B_{z_j}(a_j/10)\subset V_j$, so by Lemma \ref{lemma::harm_av_monotone} $$\mathbb{E}[(h^\D,\rho_{z_j}^{\partial V_j})^4]\leq \mathbb{E}[(h^\D,\rho_{z_j}^{\partial B_{z_j}(a_j/10)})^4].$$
Using Proposition \ref{prop:K_a.s._finite}, we see that this is less than $c\log^2(a_j)$ for some universal $c$.
\end{proof}

\subsection{Proof of Gaussianity}

The proof of Proposition \ref{prop:circle_avg_pointwise_Gaussian} is based on the following lemma. Let $D'\Subset D$ be an analytic Jordan domain \footnote{by analytic Jordan domain we mean a simply connected domain bounded by a Jordan curve, where the curve is the image of the unit circle under a conformal map defined on an open neighbourhood of the unit circle.} containing $k$ pairwise distinct points $z_1,\cdots, z_k$.

\begin{prop}
\label{lemma::analytic_harm_avg_gaussian}
$((h^D,\rho_{z_1}^{\partial D'}), \cdots, (h^D,\rho_{z_k}^{\partial D'}))$
is a Gaussian vector.
\end{prop}

By conformal invariance, we can assume for the proof that $D=\D$. To prove this we will need the following technical lemma.

\begin{lemma}
\label{lemma::approximating_analytic_domains}
Let $D'\Subset \D$ be an analytic Jordan domain containing $k$ pairwise distinct points $z_1,\cdots, z_k$.
Then there exists a sequence of \emph{increasing} domains $(D_s)_{s\in [0,1]}$ with $D_0=D'$ and $D_1=\D$, such that
\begin{itemize}
\item $D_s$ is an analytic Jordan domain for every $s\in [0,1]$.
\item $d_H(\overline{D}_s,\overline{D}_t)\leq c|s-t|$ for all $s,t\in [0,1]$ where $d_H$ is the Hausdorff distance and $c$ does not depend on $s,t\in [0,1]$.
\item If $\phi_{j,s}: \overline{D}_s\to \overline{\D}$ for each $1\leq j \leq k$ and $s\in [0,1]$ is the unique conformal map sending $z_j\mapsto 0$ and with $\phi_{j,s}'(z_j)>0$ then
\begin{equation*}
\sup_{s\in [0,1]}\sup_{1\leq j\leq k} \sup_{z\in D_s}|\phi_{j,s}'(z)| <\infty
\end{equation*}
\end{itemize}
\end{lemma}

\begin{proof}
This fact seems intuitive and may well be known but we could not find a reference. The proof we give here is elementary and relies on Brownian motion estimates as well as explicit constructions of Riemann maps.

Consider the doubly connected domain $\D\setminus \overline{D'}$. Then, by the Riemann mapping theorem for doubly connected domains (\cite[Ch6, \S 5, Theorem 10]{ahlfors}), there exists a conformal map $\phi$ from $\D\setminus \overline{D'}$ to the annulus $\D\setminus r\overline{\D}$ for some unique $r<1$. We set $$D_s:=\{\phi^{-1}((r+(1-r)s)\D\setminus r\overline{\D})\}\cup D'$$ for each $s\in [0,1]$ so that $D_0=D'$, $D_1=\D$ and the $(D_s)_{s}$ are increasing as required. It is also clear that $D_s$ is an analytic Jordan domain for each $s$.

Moreover, as $D'=D_0$ has analytic boundary we know that $\phi^{-1}$ can be extended analytically, by Schwarz reflection, to $\D\setminus u\overline{\D}$ for some $u<r$ (and we can pick $u$ such that $z_i\notin \phi^{-1}(\D\setminus u\overline{\D})$ for each $1\le i\le k$). We also have that $|\phi'|$ is a continuous function on the compact set $\overline{\D}\setminus D' $ (because $\phi$ extends analytically to $\bar{\D}\setminus {D'}$) so is bounded above and below on this set. This provides the second statement of the lemma (concerning Hausdorff distance).

For the third statement, we pick $u<v<r$ and define $V$ to be the domain given by the interior of the Jordan curve $\phi^{-1}(\partial (v\D))$. Similarly, we define $U$ to be the domain bounded by the curve $\phi^{-1}(\partial (u\D))$, so that $U\Subset V\Subset D'$. Then we set $$M=\sup_{x\in \partial U} \sup_{y\in \partial V}  K^{V}(x,y)$$
where $K^{V}(x,\cdot)$ is the boundary Poisson kernel on $\partial V$. That is, $K^{V}(x,y)$ is the density, with respect to arc length, of the harmonic measure on $\partial V$ viewed from $x\in V$. We recall here that for an analytic Jordan domain $D$, and $x\in D$, $y\in \partial D$
\begin{equation} \label{eqn::poisson_kernel_facts}
K^D(x,y)= \partial_n G^{D}(x,y) = \sqrt{2}|\varphi_x'(y)|\end{equation} for $\varphi_x:\overline{D}\to \overline{\D}$ the unique conformal map with $\varphi_x(x)=0$ and $\varphi_x'(x)>0$. \footnote{this follows from the fact that for analytic $D$, $\varphi_{x}(z)=\e^{-G^D(x,z)-i\tilde{G}^D(x,z)}$ where $\tilde{G}^D(x,\cdot)$ is the harmonic conjugate of $G^D(x,\cdot)$, see for example \cite{krantz}.} In particular, since $\partial V$ is an analytic Jordan curve, $|\varphi_{x}'(y)|$ is a continuous function on $\partial U \times \partial V$, and this means that $M$ defined above is finite.

 We will use the fact that for any $s\in [0,1]$, by definition of $\phi$ and conformal invariance, the image under $\phi$ of a Brownian motion started at $y\in \partial V$ and stopped when it leaves $D_s\setminus \overline{U}$ is a Brownian motion started at $\phi(y)\in \partial (v\D)$ and stopped when it leaves $(r+(1-r)s)\D\setminus u\overline{\D}$. We refer to this elementary fact as ($\dagger$).

First, we will use ($\dagger$) to prove that for any $z\in \partial D_s$, if $\mathbf{n}(z)$ is the inward unit normal vector to $\partial D_s$ at $z$, then
\begin{equation} \label{eqn::G_bound} \partial_n G^{D_s}(z_i,z)=\lim_{\delta\to 0} \delta^{-1} G^{D_s}(z_i,z+\delta \mathbf{n}(z)) \leq c \end{equation} where the constant $c$ is independent of $1\leq i \leq k$, $s\in [0,1]$ and $z\in \partial D_s$. To do this, without loss of generality we take $i=1$. Assume that $\delta$ is always small enough that $z+\delta\mathbf{n}(z)$ does not intersect $\overline{V}$. Then we take a Brownian motion $(B_t)_{t\geq 0}$ in $\C$ started from $z_1$, and define the following series of stopping times:
$$T_1=\inf \{t\geq 0: B_t\notin V\}; \;\; S_1=\tau_{D_s}\wedge\inf \{t\geq T_1: B_t\in \overline{U} \}, $$
$$ T_j =\tau_{D_s}\wedge\inf \{t\geq S_{j-1}:B_t \notin V\} \;\; S_j = \tau_{D_s}\wedge\inf \{t\geq T_j: B_t\in \overline{U}\}\;\; \text{for} \, j\geq 2$$
where $\tau_{D_s}$ is the hitting time of $\partial D_s$.
Then for each time interval $[T_j,S_j]$, writing $p_t$ for the transition density of Brownian motion in $\C$, we have
\begin{eqnarray*}
\E_{z_1}\left[\int_{t=T_j}^{S_j} p_t(z_1,z+\delta \mathbf{n}(z)) \, dt\right] & = & \E_{z_1} \left[ \int_{t=0}^{S_j-T_j} p_t(B_{T_j},z+\delta \mathbf{n}(z)) \, dt \right] \\
&\leq & M |\partial V| \sup_{x\in \partial (v\D)} G^{(r+(1-r)s)\D}(x,\phi(z+\delta \mathbf{n}(z)))
\end{eqnarray*} where $|\partial V|$ is the length of the curve $\partial V$. The inequality follows from ($\dagger$) since the expected time that a Brownian motion started at $x\in \partial(v\D)$ spends at any given point before exiting $(r+(1-r)s)\D\setminus \overline{U}$ is less than the expected time spent there before exiting $(r+(1-r)s)\D$. This gives us that
$$\limsup_{\delta\to 0} \delta^{-1}G^{D_s}(z_i,z+\delta \mathbf{n}(z))\leq C \limsup_{\delta\to 0} \left( \delta^{-1} \sup_{x\in \partial (v\D)} G^{(r+(1-r)s)\D}(x,\phi(z+\delta \mathbf{n}(z)))\right) \E_{z_1}[\,|\{j: S_j<\infty\}|\,].$$ Now, since $|\{j: S_j<\infty\}|$ is dominated by a geometric random variable with success probability uniformly bounded below (for example, the probability that a Brownian motion started on $\partial (v\D)$ hits $\partial \D$ before $\partial(u\D )$) we see that the expectation is bounded, independently of $z$ and $s\in [0,1]$. Thus we only need to consider the limsup term in the above. For this, we first note that $|\phi(z+\delta \mathbf{n}(z))|\geq (r+(1-r)s)(1-K\delta)$ for some $K$ depending only on $\phi$ (since $\phi$ has uniformly bounded derivative). Then, an explicit calculation using the Green's function in the unit disc tells us that
$$\sup_{x\in \partial (v\D)} G^{(r+(1-r)s)\D}(x,\phi(z+\delta \mathbf{n}(z))) \leq \log \left(1+ K\delta\left(\frac{a^2-2+K\delta-K\delta a^2}{(a(1-K\delta)-1)^2}\right)\right) $$
where $a:=v/(r+(1-r)s)$. Since $|a|\leq v/r <1$, we obtain (\ref{eqn::G_bound}).

Now recall the definition of $\phi_{j,s}$ from the statement of the lemma. We have just proved, by the second equality in (\ref{eqn::poisson_kernel_facts}), that $$
\sup_{z\in \partial D_s}|\phi_{j,s}'(z)|\leq c$$ for some $c$ not depending on $j$ or $s$. However, since $\phi'_{j,s}$ is analytic up to the boundary of $D_s$ we obtain the same upper bound for $\sup_{z\in \overline D_s}|\phi_{j,s}'(z)|.$
\end{proof}

\begin{proof}[Proof of Proposition \ref{lemma::analytic_harm_avg_gaussian}] To prove Proposition \ref{lemma::analytic_harm_avg_gaussian}, we take a sequence of increasing domains $(D_s)_{s\in[0,1]}$ as described by Lemma \ref{lemma::approximating_analytic_domains}. Then we define $$X^{\ii}_s:= (h^\D,\rho^{\partial D_{1-s}}_{z_i})$$ for all $i,s$ and let $$\b X_s := (X^{(1)}_{s},\ldots, X^{(k)}_{s})$$ (note the reversal of time here - we want to now move inwards from $\partial \D$ to $\partial D_s$).
We will prove that for every $s$, $\b X_s$ is distributed as a multivariate Gaussian random vector. Setting $s=1$, this proves the lemma.

In fact, we will prove the following equivalent statement: for every vector $(a_1,\ldots, a_k) \in \R^k$, and $s>0$
\begin{equation*}
Y_s:= \sum_{i=1}^k a_i X^{(i)}_{s}
\end{equation*}
is a Gaussian random variable. Note that $Y_0=0$ because $h^\D$ has zero boundary conditions, and it is also straightforward to check using the domain Markov property that $Y_s$ has independent mean zero increments. By the Dubins--Schwarz theorem, these observations tell us that as long as $Y_s$ has a continuous modification, it must be a Gaussian process (because it is a continuous martingale with deterministic quadratic variation process).

To prove that $Y_s$ has continuous modification, we shall prove that for any $\eta>0$ there exists some constant $C$ such that for all $\ve>0$ and $s\in [0,1]$
\begin{equation}
\E[(Y_{s} - Y_{s+\ve})^4] \le C\ve^{2-\eta}\label{eq:fourth_moment}
\end{equation}
Using Kolmogorov's continuity criterion, \eqref{eq:fourth_moment} is enough to conclude that $Y_s$
admits a continuous modification.

Fix some $0<\eta<1$, let $s\in [0,1)$ and let $\gamma_\eps$ be the curve defined by $\partial D_{1-s-\eps}$ inside $D_{1-s}$. Then by definition, expansion and Cauchy-Schwarz,
$$\E[(Y_s-Y_{s+\eps})^4]=\E[\big(\sum_{i=1}^k a_i(h^{D_{1-s}},\rho_{z_i}^{\gamma_\eps})\big)^4]\leq  \sum_{1 \le i_1 \le \ldots \le i_4 \le k} a_{i_1} a_{i_2}a_{i_3} a_{i_4} \prod_{j=1}^4 \Big( \E[(h^{D_{1-s}},\rho^{\gamma_{\eps}}_{z_{i_j}})^4 ] \Big)^{1/4}.$$

 In light of the above inequality, it is enough to show that there exists a $C$ such that for all $1\le j\le k$, $s\in[0,1]$ and $\eps>0$
\begin{equation}
\E\big[ (h^{D_{1-s}},\rho_{z_j}^{\gamma_\eps})^4\big]  \le  C\ve^{2-\eta}.\label{eq:four_integral_bound}
\end{equation}

For this, we use our hypotheses on the family of domains $(D_s)_{0\leq s \leq 1}$. These tell us that if $\phi_{j,s}:D_s \to \D$ is the unique conformal map sending $z_j\mapsto 0$ and with $\phi'_{j,s}(z_j)>0$, we have that $\phi_{j,s}(\gamma_{\eps})$ is contained in $\{z: 1-b\eps<|z|<1\}$ for some $b>0$ not depending on $j,s$ or $\eps$. Then by conformal invariance, we can write
$$ \E[(h^{D_{1-s}},\rho_{z_j}^{\gamma_{\eps}})^4]=\E[(h^\D,\rho_0^{\phi_{j,s}(\gamma_{\eps})})^4]\leq \E[(h^\D,\rho_0^{\partial (1-b\eps)\D})^4],$$
where the inequality follows from Lemma \ref{lemma::harm_av_monotone}.

So we estimate the final quantity; without loss of generality, we assume that $b=1$. By Fatou's Lemma we have
\begin{equation}
\label{eqn::gaussianity_integral}
\E[(h^\D,\rho_0^{\partial (1-\eps)\D})^4]\leq \liminf_{\delta\to 0} \iint_{\D^4} \tilde K_4^\D(x_1,x_2,x_3,x_4) \hat\psi^{\delta}(x_1)\hat\psi^\delta(x_2)\hat\psi^\delta(x_3)\hat\psi^\delta(x_4)\, \prod dx_i,\end{equation}
recalling the definition of $\hat\psi^\delta$ from Section \ref{sec::harm_avg}: it is a smooth function, bounded above by some constant multiple of $\delta^{-1}$, that is supported on the annulus $\Omega_{\delta,\eps}:=\{z: (1-\eps)(1-11\delta/10) \leq |z| \leq (1-\eps)(1-4\delta/10) \}$. Suppose that $\delta<\eps$. Then the integrand is only supported on points $(x_1,x_2,x_3,x_4)$ all lying in $\Omega_{\delta,\eps}$. Moreover, if $(x_1,\cdots, x_4)$ are 4 such points, then by Proposition \ref{lemma::voronoi_estimates} $$\tilde{K}_4^\D(x_1,\cdots, x_4) \leq \big(\prod_{j=1}^4 \mathbb{E}[(h^\D,\rho_{z_j}^{\partial V_j})^4]\big)^{1/4}\leq c\big( \prod_j (\frac{\eps^4}{a_j^4}\wedge 1)\log^4(a_j)\big)^{1/4}$$ for some universal constant $c$, where $a_j=\min_{i\ne j}\{|u(x_i)-u(x_j)|_{1}\}$ (and $u(x)=x/|x|$.) Using the bound on $\hat \psi^\delta$, we see that (\ref{eqn::gaussianity_integral}) is bounded above by
\begin{equation}\label{eqn::gaussianity_int_2} c'\liminf_{\delta\to 0} \delta^{-4}\iint_{\Omega_{\delta,\eps}^4} \left(\prod_{j=1}^4 (\frac{\eps^4}{a_j^4}\wedge 1)\log^4(a_j)\right)^{1/4} \, dx_j\end{equation} for another universal $c'$.

Now, we rewrite the integral in polar coordinates $x_j=r_j\e^{i\theta_j}$ (so $u_j=\e^{i\theta_j}$) and then, noticing that $a_j$ depends only on the angular coordinate, integrate over $r_1,\cdots, r_4$. This gives us that (\ref{eqn::gaussianity_int_2}) is less than or equal to
$$c'' \iint_{0\leq \theta_1\leq \theta_2 \leq \theta_3 \leq \theta_4 \leq 2\pi} \left(\prod_{j=1}^4 (\frac{\eps^4}{a_j^4}\wedge 1)\log^4(a_j)\right)^{1/4} \, d\theta_j$$
where $a_j=a_j(\theta_1,\cdots, \theta_4).$ Now, we divide the integral over the $\theta_j$'s into several parts, depending on which $a_j$'s are smaller or bigger than $\eps$. Let
\begin{align}
(A_j)_{1 \leq j\leq 4} & := \{ a_{(j+k)\text{mod} 4} <\eps \text{ for } k=0,1; \;a_{(j+k)\text{mod} 4} \ge\eps \text{ for } k=2,3\} \nonumber\\
(B_j)_{1\leq j \leq 4} & := \{ a_{(j+k)\text{mod} 4} <\eps \text{ for } k=0,1,2; \; a_{(j+3)\text{mod} 4} \ge\eps \}  \nonumber\\
D& := \{ a_j <\eps \text{ for } j=1,2,3,4\},\nonumber \;\; \text{and} \\
E& := \{ a_j \ge\eps \text{ for } j=1,2,3,4\}.\nonumber
\end{align}
A computation yields that the integral over $A_j$ is $O(\eps^{2-\eta})$ for all $j$, the integral over $B_j$ is $O( \eps^{3-\eta})$ for all $j$, the integral over $D$ is $O(\eps^{2-\eta})$, and the integral over $E$ is $O(\eps^{2-\eta})$. This completes the proof of equation \ref{eq:four_integral_bound} and hence the lemma.
\end{proof}

\begin{proof}[Proof of \cref{prop:circle_avg_pointwise_Gaussian}]
The strategy of the proof is to construct a sequence of analytic Jordan domains $(D_n)_{n\geq 1}$, all contained in $D$, such that $((h^D,\rho_{z_1}^{D_n}),\cdots, (h^D,\rho_{z_k}^{D_n})) \to (h_\ve(z_1) ,\ldots, h_\ve(z_k))$ in a precise sense as $n\to \infty$.  More concretely, it is enough to show that for any $(a_1,\cdots, a_k)\in \R^k$, we can choose a sequence of analytic domains $(D_n)_{n\geq 1}$, such that setting $$Y_n:=\sum_{i=1}^k a_i (h^D,\rho_{z_i}^{D_n})  \;\; \text{and} \;\; Z:=\sum_{i=1}^k a_i h_\eps^{D}(z_i),$$ we have
\begin{equation}
\label{eqn::varto0}
\var(Y_n-Z)\to 0
\end{equation} as $n\to \infty$. Since $Y_n$ is Gaussian for every $n$ (by Proposition \ref{lemma::analytic_harm_avg_gaussian}) and $Z$ has finite variance, this shows that $Z$ is Gaussian.

So, we choose the $D_n$. This will involve first defining a sequence of auxiliary domains $D_n'$, that need not be analytic, and then using them to define the analytic domains $D_n$.

To begin, we observe that for $n\in \N$ with $1/n<\eps$, the balls $\{B_{z_i}(\eps+1/n): 1\le i \le k\}$ are disjoint. Let us choose a further point $z\in D$, that does not lie in any of these balls. It is easy to see (since the set $\{z_i\}_{1\le i \le k}$ is finite) that one can choose such a $z$, along with a smooth curve $\gamma_i$ from $z$ to $z_i$ for each $1\le i \le k$, and $c,c'\in (0,1)$ such that: 
\begin{itemize}
	\item $\gamma_i\cap \partial B(z_j,\eps)$ is empty for $i\ne j$ and consists of exactly one point when $i=j$, $1\le i \le k$; 
	\item the $c/n$ fattenings $\gamma_i^n:= \{z\in D: d(z,\gamma_i)<c/n\}$ of the $\gamma_i$ are such that $D'_n:=\bigcup_{1\le i \le k} \gamma_i^n \cup B_{z_i}(\eps+1/n)$ is a simply connected domain strictly contained in $D$ for every $n>1/\eps$;
	\item the boundary of $D_n'$ contains, for each $1\leq i \leq k$, the curve $\partial B_{z_i}(\eps+1/n)\setminus A_i^n$, where $A_i^n$ is an arc of $\partial B_{z_i}(\eps)$ that has length $\le c'/n$.
	\end{itemize}
We need the following basic statement that says, in some sense, that $D_n'$ is a good approximation to $\cup_i B_{z_i}(\eps)$ for large $n$.
\begin{lemma}\label{lem:dnclose}
	For every $1\le i \le k$, $$\sup_{D'_n\subset D'\subset D_{n/2}'}\iint G^{D'}(x,y) \rho_{z_i}^\eps(dx)\rho_{z_i}^\eps(dy)\to 0$$ as $n\to \infty$, where the supremum is over all simply connected domains $D'$ satisfying the indicated inclusions.
\end{lemma}
\begin{proof}
Without loss of generality, we prove the result for $i=1$, and assume that $\text{diam}(D)\le 1$. 
Fix $D_n'\subset D' \subset D_{n/2}'$ simply connected. Then by harmonicity of the Green's function, we have 
\[ \iint G^{D'}(x,y) \rho_{z_1}^\eps(dx)\rho_{z_1}^\eps(dy)=\int G^{D'}(z_1,y) \rho_{z_1}^\eps(dy)\]
and also for $y\in \partial B_{z_1}(\eps)$, 
\begin{equation}
\label{eqn::greensmonbound}0\le G^{D'}(z_1,y)\le G^{D''}(z_1,y)=\log(1/\eps)-\mathbb{E}_y[\log(1/|B_{\tau_{D''}}-z_1|)],\end{equation} where the expectation $\mathbb{E}_y$ is for a Brownian motion $B$ starting from $y$, and $\tau_{D''}$ is its exit time from $D'':=D_{n/2}'$. 

Moreover, we have the upper bound $\mathbb{E}_y[\log(1/|B_{\tau_{D''}}-z_1|)]\ge \log(1/(\eps+2/n))(1-p_{y,n}),$ where $p_{y,n}$ is the probability that a Brownian motion started from $y$ exits $\partial B_{z_1}(\eps+2/n)$ through the boundary arc $A_1^{n/2}$. Since $p_{y,n}$ tends to $0$ as $n\to \infty$ for almost every $y\in B_{z_1}(\eps)$ (in fact, the only $y$ for which this fails to hold is the single intersection point of $\partial B_{z_1}(\eps)$ and $\gamma_1$), it follows by dominated convergence that $\int G^{D''}(z_1,y)\rho_{z_1}^\eps(dy)\to 0$ as $n\to \infty$. The lemma then follows from the inequalities on the left hand side of  \eqref{eqn::greensmonbound}.
\end{proof}
\begin{figure}[htbp]
	\begin{center}
		\includegraphics[scale = 0.5]{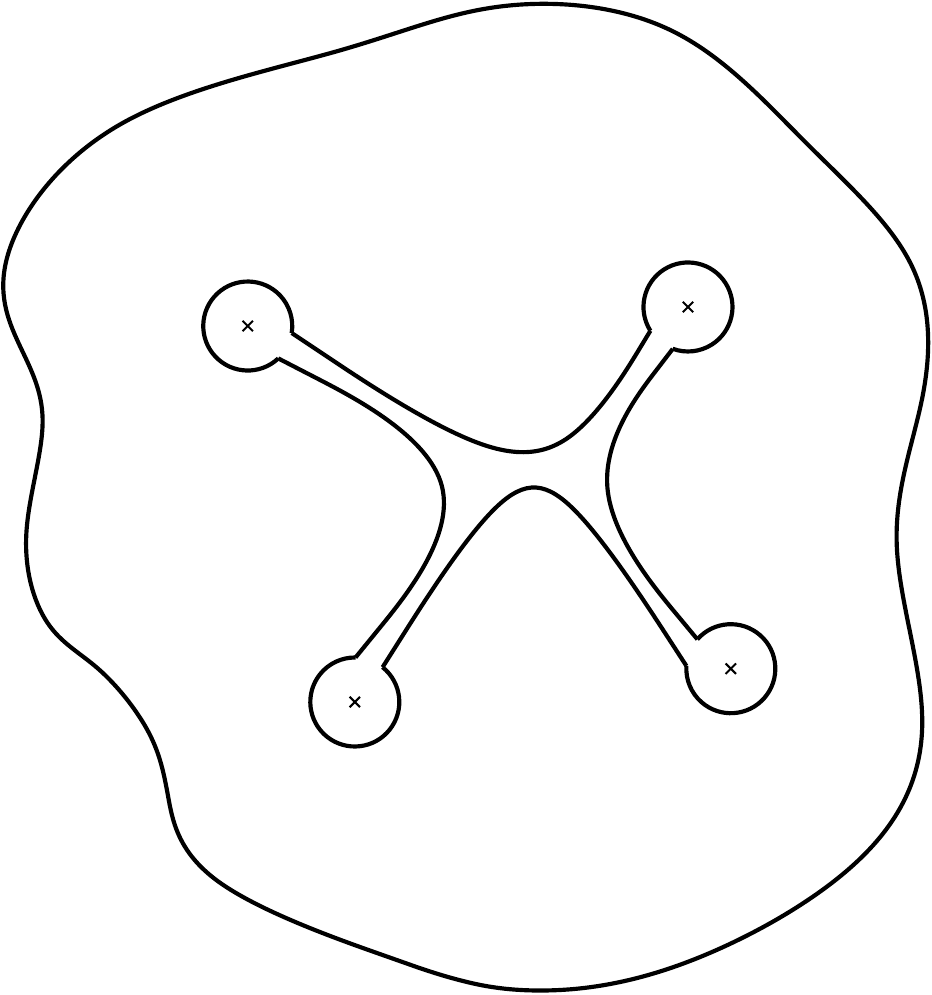}
		\caption{The domain $D_n'$ from the proof of \cref{prop:circle_avg_pointwise_Gaussian}. The boundary of $D_n'$ overlaps with $\partial B_{z_i}(\eps+1/n)$ for $1\leq i \leq k$ except on k small arcs $(A_i^n)_{1\le i\le k}$ with maximum length tending to $0$. Here $k=4$.}
		\label{fig:domain}
	\end{center}
\end{figure}
Now from the $(D_n')_n$ we define our sequence of domains $(D_n)_n$, such that $D_n$ is analytic, and also $D_n'\subset D_n\subset D_{n/2}'$ for each $n$. This second condition will allow us to apply Lemma \ref{lem:dnclose}. 

By the Riemann mapping theorem for doubly connected domains, we know that we can choose a conformal map $\phi$ from $D\setminus \overline{D_n'}$ to the annulus $\D\setminus \overline{r\D}$ for some unique $r\in (0,1)$.  For each $r<s<1$, denote by $D_n'(s)$ the complement in $D$ of the preimage of $\D\setminus \overline{s\D}$ under $\phi$. Then $D_n'(s)$ is a simply connected domain containing $D_n'$ for every $s\in (r,1)$, and $\cap_{s\in (r,1)}D_n'(s)$ is equal to $D_n'$. Hence there exists some $1>s_n>r$ such that $D'_n(s)$ is contained in $D_{n/2}'$. We then define $$D_n:= D'_n(s_n).$$
It is clear that $D_n$ is analytic for every $n$ (since by definition its boundary is the image of the unit circle under a conformal map that is defined in a neighbourhood of the circle) and also, by construction, that $D_{n/2}'\subset D_n\subset D_n'$. 

Having defined the $D_n$, we just need to prove (\ref{eqn::varto0}). Without loss of generality it is enough to show that
	$\mathbb{E}[(h^D_\eps(z_1)-(h^D,\rho_{z_1}^{\partial D_n}))^2]\to 0$ as $n\to \infty$. For this, write $h^D=h_D^{D_n}+\varphi_D^{D_n}$ using the domain Markov decomposition, so that $(h^D,\rho_{z_1}^{\partial D_n})=\varphi_D^{D_n}(z_1).$ Then since $B_{z_1}(\eps)\subset D_n$ we can further write $h_D^{D_n}=h_{D_n}^{B_{z_1}(\eps)}+\varphi_{D_n}^{B_{z_1}(\eps)}$, and by uniqueness, we must have $$h^D_\eps(z_1)=\varphi_{D_n}^{B_{z_1}(\eps)}(z_1)+\varphi_D^{D_n}(z_1).$$ Thus, we need to show that $\mathbb{E}[\varphi_{D_n}^{B_{z_1}(\eps)}(z_1)^2]\to 0$ as $n\to \infty$. However, from the definition of the circle average as an $L^2$ limit (Lemma \ref{lemma::circle_average_defn}) and the identification of the covariance structure \eqref{eqn:varhDphi}, we know that $$\mathbb{E}[\varphi_{D_n}^{B_{z_1}(\eps)}(z_1)^2]=\iint G^{D_n}(x,y)\rho_{z_1}^\eps(dx)\rho_{z_1}^\eps(dy).$$  The result then follows from Lemma \ref{lem:dnclose}.

\end{proof}

\section{Proof of Theorem \ref{thm::characterisation_gff}}\label{sec:Concl}
To conclude we prove convergence of the circle average field, which then implies Theorem \ref{thm::characterisation_gff} by Lemma \ref{lemma::circ_avg_good_approx}.

\begin{lemma}\label{lemma::conv_circ_avg_field}
	For any $\phi\in C_c^\infty (D)$, $(h_\eps^D,\phi)$ converges to $(h^D_{\text{GFF}},\phi)$ in distribution as $\eps\to 0$.
\end{lemma}

We first see how this implies Theorem \ref{thm::characterisation_gff}.

\begin{proof}[Proof of Theorem \ref{thm::characterisation_gff}]
	To prove that $h^D\overset{(d)}{=} h^D_{\text{GFF}}$ we need to show that for any $(\phi_1,\cdots, \phi_n)$ with $(\phi_i)_{1\leq i \leq n}\in C_c^\infty(D)$, $((h^D,\phi_1), \cdots, (h^D,\phi_n))$ is a Gaussian vector with mean $0$ and the correct covariance matrix. Equivalently, we need to show that for any $(u_1,\cdots, u_n)\in \R^n$, the sum $\sum_1^n u_i (h^D,\phi_i)$ is a centered Gaussian variable with the correct variance. By linearity, we therefore need only prove that for any $\phi\in C_c^\infty(D)$, $$(h^D,\phi)\overset{(d)}{=}(h^D_{\text{GFF}},\phi).$$
	
	So, we fix such a $\phi$. By Lemma \ref{lemma::circ_avg_good_approx}, we know that $\Var((h_\eps^D,\phi)-(h^D,\phi))\to 0$ as $\eps \to 0$. Thus $(h_\eps^D,\phi)$ converges to $(h^D,\phi)$ in distribution. From here, Lemma \ref{lemma::conv_circ_avg_field} implies the result.
	\end{proof}

\begin{proof}[Proof of Lemma \ref{lemma::conv_circ_avg_field}.]
	We will prove that for every $n\in \N$, $\mathbb{E}[(h_\eps^D,\phi)^n]\to \mathbb{E}[(h_{\text{GFF}}^D,\phi)^n]$ as $\eps\to 0$, which implies the result by the method of moments, \cite[Theorem 30.2]{billingsley}. This requires a bit of care however, since a priori it is not even clear that this moment is well defined when $n \ge 4$. 
	
	To show the convergence, we need to compute the limit as $\eps\to 0$ of
	\begin{equation}\label{eqn::moments_circ_avg}
	\mathbb{E}[(h_\eps^D,\phi)^n] =  \mathbb{E}\left[\iint_{A_\eps} \big(\prod_{i=1}^n h_\eps^D(z_i)\phi(z_i) dz_i\big)\right] +   \mathbb{E}\left[\iint_{E_\eps} \big(\prod_{i=1}^n h_\eps^D(z_i)\phi(z_i) dz_i\big)\right]=:I^{A_\eps}+I^{E_\eps},
	\end{equation}
	where in the middle term, we have decomposed the integral over $D^n$ into the integrals over $A_\ve  := \{ (z_1,\ldots, z_n) \in D^n: |z_i- z_j|>2\ve \text{ for all } i,j\}$ and $E_\ve := D^n \setminus A_\ve$. We assume that $\eps>0$ is always small enough that $d(z,\partial D)>2\eps$ for every $z$ in the support of $\phi$. We will consider the right hand side and show that both terms are well defined and finite, from which it will follow by Fubini's theorem that the moment on the left hand side is also finite. 
	
Let us first show that $I^{E_\eps}\to 0$ as $\eps\to 0$. This follows from our a priori bounds on the two point function in \cref{Variance}. Indeed, for any $z_1,\ldots, z_n$ in the support of $\phi$, we have that the $h_\ve(z_i)$  are marginally Gaussian (Proposition \ref{prop:circle_avg_pointwise_Gaussian}), and therefore the $n$th moment of $|h_\eps (z_i)|$ is at most $c \E( h_\eps(z_i)^2)^{n/2}$ for some constant $c$ depending only on $n$. Therefore by H\"{o}lder's inequality and \cref{Variance}, we have
	$$\mathbb{E} (\big(\prod_{i=1}^n |h_\eps^D(z_i|)) \le c(\log (1/\ve))^{n/2},$$
	for some constant $c$ depending on $n$ but not on $\eps$ (note this already implies that for fixed $\eps>0$, $(h_\eps, \phi)$ has finite $n$th moment). Hence we can apply Fubini to bring the expectation inside the integral in $I^{E_\eps}$, and conclude that
\begin{equation} \label{eqn:Eeps0}
\lim_{\ve\to 0}	I^{E_\eps} = \lim_{\ve \to 0} \, \left[\iint_{E_\ve} \mathbb{E} \big(\prod_{i=1}^n| h_\eps^D(z_i)\phi(z_i)| dz_i\big)\right] \le \lim_{\ve \to 0} c(\log (1/\ve))^{n/2} \ve^2 = 0.
	\end{equation}
Here we have used that the integral of $\prod_i |\phi(z_i)|$ over $E_\ve$ is $O(\ve^2)$: indeed the $n$-dimensional volume of $E_\ve$ is $O(\ve^2)$ by definition for fixed $n \ge 2$, and $\phi$ is bounded. 

Consequently, we need only consider the term $I^{A_\eps}$ on the right-hand side of \eqref{eqn::moments_circ_avg}.	
For this we use Proposition \ref{prop:circle_avg_pointwise_Gaussian}, which tells us that for every $(z_1,\cdots, z_n)\in A_\eps$ , $(h_\eps^D(z_1),\cdots, h_\eps^D(z_n))$ is multivariate normal with mean $(0, \ldots, 0)$. Therefore, by the Wick rule (to be more precise, Isserlis' theorem), we have that \begin{equation}
\label{eqn::wick}\mathbb{E}\big[\prod_{i=1}^n h_\eps^D(z_i)\big]= \mathbf{1}_{n \, \text{even}}\sum_{P:\text{pairings}}  \prod_{(i,j) \in P}\mathbb{E}[h^D_\eps(z_i)h^D_\eps(z_j)], \, \end{equation}
on $A_\eps$,
where the above sum is over all pairings of $\{1,2,\ldots, n\}$. In fact, by (\ref{eqn::lim_circav_kernel}), Proposition \ref{prop:K_a.s._finite}, Lemma \ref{lemma::variance_harm_avg} and Cauchy--Schwarz, we know that for any $z_i,z_j \in A_\ve$,

\begin{equation}
	\label{eqn::final_cov_bounds}
	|\E[h_\eps^D(z_i)h_\eps^D(z_j)] |\leq \log\left(\frac{1}{ |z_i-z_j|}\right)+O(1).
	\end{equation}  This allows us to deduce that the right hand side of (\ref{eqn::wick}) is bounded above by a function independent of $\ve$, that  is also integrable over $D^n$. Thus we can apply Fubini and then the dominated convergence theorem in (\ref{eqn::moments_circ_avg}), to see that \begin{align*}
	\lim_{\ve \to 0}\mathbb{E}[(h_\eps^D,\phi)^n]  =	\lim_{\ve\to 0} I^{A_\eps}
	& =  \mathbf{1}_{n \, \text{even}} \; \lim_{\eps\to 0} \,  \iint_{A_\ve}  \sum_{P:\text{pairings}}  \prod_{(i,j) \in P}\mathbb{E}[h^D_\eps(z_i)h^D_\eps(z_j)]  \prod_{j} \phi(z_j) dz_j \\
	& =  \mathbf{1}_{n \, \text{even}} \; \iint_{A_\ve} \, \lim_{\ve  \to 0}  \sum_{P:\text{pairings}}  \prod_{(i,j) \in P}\mathbb{E}[h^D_\eps(z_i)h^D_\eps(z_j)]  \prod_{j} \phi(z_j) dz_j \\
	& =  \mathbf{1}_{n \, \text{even}} \iint_{A_\eps}
\sum_{P:\text{pairings}}  \prod_{(i,j) \in P}G^D(z_i,z_j) \prod_{j} \phi(z_j) dz_j \\
& =  \mathbf{1}_{n \, \text{even}} \iint_{D^n}
\sum_{P:\text{pairings}}  \prod_{(i,j) \in P}G^D(z_i,z_j) \prod_{j} \phi(z_j) dz_j
\end{align*} where the penultimate line follows by \eqref{eqn::lim_circav_kernel}, and the final line by the same reasoning as in \eqref{eqn:Eeps0}. From this, it follows that that $\lim_{\ve\to 0} \mathbb{E}[(h_\eps^D,\phi)^n]= \mathbb{E}[(h_{\text{GFF}}^D,\phi)^n]$, and hence we have concluded the proof of Lemma \ref{lemma::conv_circ_avg_field} and of Theorem \ref{thm::characterisation_gff}.
	\end{proof}

\section{Proof of Theorem \ref{thm:BB}} \label{sec:BB}

First, we prove that the family $\{X^{[0,2^n]}(1)\}_{n\in \N}$ is tight:
	
	\begin{lemma}\label{lemma::tightness}
	For any $\eps>0$ there exists $M>0$ such that $\P(X^{[0,2^n]}(1)\ge M)\le \eps$ for all $n\in \N$.
	\end{lemma}
	
\begin{proof}
	First observe that $X^{[0,2^n]}(1)\eqd 2^{\frac{n}{2}} X^{[0,1]}(2^{-n})$, by the assumption of Brownian scaling. Then, by iteratively dividing the interval $[0,1]$ into two and using scaling and the Markov property again, we can write
	\begin{equation}
	\label{eqn::domain_m_decomp}
	2^{\frac{n}{2}}X^{[0,1]}(2^{-n}) \eqd 2^{-\frac{n}{2}+1} \sum_{k=0}^{n-1}2^{\frac{k}{2}}X^{[0,1]}_k(1/2),
	\end{equation}
	where the $(X^{[0,1]}_k \, : \, 0\leq k \leq n-1)$ are independent copies of $X^{[0,1]}$. Write $Y_n$ for the right hand side of \eqref{eqn::domain_m_decomp}, and let $X$ have the law of $X^{[0,1]}(1/2)$. By Assumptions \ref{ass:BB} we know that
	\begin{equation}
	\label{eqn::poly_tails_X}
\mathbb{E}[\log^+(|X|)]<\infty.
	\end{equation} The idea is to derive a uniform bound (in $n$) for $\P(|Y_n|>M)$ by recursion.
	
	To do this, write \[Y_{n+1}=\frac{Y_n}{\sqrt{2}} + 2X_n^{[0,1]}(1/2)\] (where $X_n^{[0,1]}(1/2)$ has the same distribution as $X$). This means that if we pick some $a\in (1,\sqrt{2})$ and set $b=1-\frac{a}{\sqrt{2}}\in (0,1)$ we have that
	\[ \mathbb{P}(|Y_{n+1}|\ge M) \leq \mathbb{P}(|Y_n|\geq aM)+\P(|X|\ge \frac{b}{2}M).\]
	Since $\mathbb{P}(|Y_0|\ge M)=\P(|X|\ge M/2)$ we have by iteration that
	\[ \P(|Y_n|\geq M)\le \P(|X|\ge a^n \frac{M}{2}) + \sum_{k=0}^{n-1} \P(|X|\ge a^k b \frac{M}{2}) \le \sum_{k=0}^n \P(|X|\ge a^k b \frac{M}{2}),\]
	and we can bound this sum above by
	\[\sum_{k=1}^\infty \P\left(\log^+\left(\frac{2|X|}{bM}\right)\ge k \log a\right) \le \frac1 {\log a}\int _0^\infty\P\left(\log^+\left(\frac{2|X|}{bM}\right) \ge t\right) dt= \frac1 {\log a}\E\left[\log^+\left(\frac{2|X|}{bM}\right)\right].\]
	By \eqref{eqn::poly_tails_X} the right hand side converges to $0$ as $M\to \infty$, and it is clearly uniform in $n$, which completes the proof.
\end{proof}

We now claim that, locally, the process $X^{[0,2^n]}$ (in the large $n$ limit) has to be  a constant times a Brownian motion.

\begin{lemma}\label{lem:local_BM}
	We have the following convergence in the sense of finite dimensional distributions:
	$$
	(X^{[0,2^n]}(t))_{t\in[0, 1]} \xrightarrow[(d)]{n \to \infty} \sigma(B(t))_{t \in [0,1]}
$$
	for some constant $\sigma \ge 0$ where $(B(t))_{t\geq 0}$ is a standard Brownian motion.
\end{lemma}
\begin{proof}
	Step one is to show that for any sequence of natural numbers going to infinity, there exists a subsequence $n(k)$ such that $(X^{[0,2^{n(k)}]}(t))_{t\in[0, 1]}$ converges as $k \to \infty$ (in the sense of finite-dimensional distributions). To do this, we write by the domain Markov property applied to the subinterval $[0,1] \subset [0,2^n]$:
	\begin{equation}\label{eqn::local_decomposition_1d}
		(X^{[0,2^n]}(t))_{t\in[0, 1]} = (\tilde{X}^{[0,1]}(t) + tX^{[0,{2^n}]}(1)) _{t\in[0, 1]},
	\end{equation}
	where $\tilde{X}^{[0,1]}$ is an independent copy of $X^{[0,1]}$.
This means that to show {convergence of (all) the finite dimensional distributions of $X^{[0,2^n]}$ along (the same) subsequence}, it suffices to show that $X^{[0,2^n]}(1)$ has subsequential limits. However, this is just a consequence of Lemma \ref{lemma::tightness}.

So now assume that we have a subsequence $(n(k): k\ge 1)$ such that $(X^{[0,2^{n(k)}]}(t))_{t\in[0, 1]}$ converges to $(Y(t))_{t\in[0, 1]}$ in law for finite-dimensional distributions. If we can show that $Y_t = \sigma B_t$ in the sense of finite-dimensional distributions for some $\sigma\ge 0$ (not depending on the subsequence), then we will have completed the proof.

\medskip We first show that $Y$ has independent and stationary increments. Pick $0=t_0\leq t_1\leq \cdots \leq t_l \leq t_{l+1} =1$, and observe that by the Markov property,
	\begin{equation}
		\label{eqn::law_subseq_lim}
		\big(Y(t_1),Y(t_2)-Y(t_1),\cdots, Y(t_l)-Y(t_{l-1}), Y(1)-Y(t_l)\big)\end{equation} is a limit in distribution as $k\to \infty$ of
	$$
\big(X^{[0,2^{n(k)}]}_1(t_1)\,,\, \frac{t_2-t_1}{2^{n(k)}-t_1}Z^{k}_1 + X^{[0,2^{n(k)}-t_1]}_2(t_2-t_1)\,, \cdots,\frac{1-t_l}{2^{n(k)}-t_l} Z^{k}_l +X^{[0,2^{n(k)}-t_l]}_{l+1}(1-t_l) \big)
$$
	where the $X^{[0,2^{n(k)}-t_{j-1}]}_j$ (for $1\leq j\leq l+1$) are independent copies of $X^{[0,2^{n(k)}-t_{j-1}]}$ ; and
$Z^{k}_j$ for $2\le j \le l$ is defined recursively by
$$
Z^{k}_1=X^{[0,2^{n(k)}]}_1(t_1)\;\; ; \;\; Z^{k}_j=\sum_{i=1}^{j-1} \frac{t_{i+1}-t_i}{2^{n(k)}-t_i}Z^{k}_i + \sum_{i=1}^{j} X^{[0,2^{n(k)}-t_{i-1}]}_i(t_i-t_{i-1}) \text{ for } 2\le j\le l+1.
$$
	
	Now we claim that for any $s\in [0,1]$ and $u\in [0,1)$, $X^{[0,2^{n(k)}-s]}(u)$ converges in distribution as $k\to \infty$ to the same limit as $X^{[0,2^{n(k)}]}(u)$. To see this, we write by scaling and \eqref{eqn::local_decomposition_1d}, whenever $k$ is large enough that $2^{n(k)}(2^{n(k)}-s)^{-1}u\le 1$:
	\begin{align*}
	X^{[0,2^{n(k)}-s]}(u) & \overset{(d)}{=} \sqrt{\frac{2^{n(k)}-s}{2^{n(k)}}} X^{[0,2^{n(k)}]}\left(\frac{2^{n(k)}}{2^{n(k)}-s} u\right) \\
	& \overset{(d)}{=} \sqrt{\frac{2^{n(k)}-s}{2^{n(k)}}} \left( \tilde{X}^{[0,1]}\left(\frac{2^{n(k)}}{2^{n(k)}-s} u\right)+\left(\frac{2^{n(k)}}{2^{n(k)}-s} u\right)X^{[0,2^{n(k)}]}(1)\right)
	\end{align*}
	where $\tilde{X}^{[0,1]}$ is an independent copy of $X^{[0,1]}$. Since $X^{[0,1]}$ is stochastically continuous, the claim follows.
	
	By the above claim, an induction argument, and the fact that $(t_{j+1}-t_j)/(2^{n(k)}-t_j) \to 0$ as $k\to \infty$, it follows that $(t_{j+1}-t_j)/(2^{n(k)}-t_j) \times Z_j^k$ converges to $0$ in distribution as $k\to \infty$ for every $1\le j\le l$. This means that
 the law of (\ref{eqn::law_subseq_lim}) is the same as the limit in distribution of
	$$\big(X^{[0,2^{n(k)}]}_1(t_1)\,,\,  X^{[0,2^{n(k)}]}_2(t_2-t_1)\,, \cdots, X^{[0,2^{n(k)}]}_{l+1}(1-t_l) \big).$$
	For this last step we have also used the independence of the $(X_j)$, the fact that the $(t_{j+1}-t_j)/(2^{n(k)}-t_j)\times Z_j^k$ actually converge in probability (because they converge in distribution to a constant), and the claim one more time.
	
 Finally, by independence of the $X_j$ again, we deduce that the entries in \eqref{eqn::law_subseq_lim} (and so the increments of $Y$) must be independent.
	Furthermore the distribution of the $j$th entry depends only on $t_j-t_{j-1}$ and so the increments are stationary. Hence, $(Y(t))_{t\in [0,1]}$ has independent and stationary increments. $Y$ is also continuous in probability at every $t$, because of (\ref{eqn::local_decomposition_1d}) and Assumptions \ref{ass:BB}. Thus $Y$ is a L\'evy process on $[0,1]$ (and can be extended to a L\'evy process on all of $[0, \infty)$ by adding independent copies on $[1,2], [2,3], \ldots $).
	
	Now it is clear that $Y$ also enjoys the scaling property: for $t \le 1$,
\begin{equation*}
		Y(t) = \lim_{k \to \infty}  X^{[0,2^{n(k)}]}(t)  = \lim_{k\to \infty} \sqrt{t}X^{[0,2^{n(k)}/t]}(1)  =  \sqrt{t}Y(1)
	\end{equation*}
	where all the equalities above are in law and the limits are in the sense of distribution.
To justify the last equality we write, by the domain Markov property,
		\[
		\sqrt{t} X^{[0,2^{n(k)}/t]}(1)\overset{(d)}{=} \sqrt{t} \tilde{X}^{[0,2^{n(k)}]}(1)+\sqrt{t}2^{-n(k)}X^{[0,2^{n(k)}/t]}(2^{n(k)}), \] where $X$ and $\tilde{X}$ are independent. Since the first term converges to $\sqrt{t}Y(1)$ in distribution, and the second, by scaling, is equal in distribution to $2^{-n(k)/2} X^{[0,1]}(t)$, we obtain the result.
	
Because $Y$ is a L\'evy process, we know that for any $\theta \in \R$, the characteristic function of $Y$ can be written as
	$$
	\E[e^{i\theta Y(t)}] = e^{\psi(\theta)t}.
	$$
(In fact, by the L\'evy--Khinchin theorem, $\psi$ has an explicit representation which will not be required here). By scaling,
$$
\E[e^{i \theta Y(t)}] = \E[e^{i \theta \sqrt{t} Y(1)}]
$$
so that
$$
t \psi (\theta) = \psi (\sqrt{t} \theta)
$$
for any $\theta >0$ and any $t \ge 0$. Set $\sqrt{t} \theta =1$ so that $t = 1/\theta^2$. Then we deduce that
$$
\psi(\theta) = \theta^2 \psi(1)
$$
for all $\theta>0$. Since $| \E[e^{i \theta Y(t)} ] | \le 1$ we see that $\psi(\theta) \le 0$ and hence it follows that $Y$ is a multiple of Brownian motion. (While we only know the characteristic function in the positive half-line, this is enough to compute the moments and check that this matches with those of a Gaussian random variable). In other words, $Y $ is $\sigma$ times a standard Brownian motion.

The final thing to check is that $\sigma$ does not depend on the subsequence along which we assumed convergence. We first argue that, {for any fixed $t \in [0,1]$}, $X^{[0,1]}(t)$ has Gaussian tails and thus has moments of arbitrary order. Applying the Markov property, $Y$ is the limit in distribution as $k\to \infty $ of $(X^{[0,1]}(t)+tX^{[0,2^{n(k)}]}(1))_{t\in [0,1]}$, where the two terms on the right are independent. Hence we can write
\begin{equation}\label{relYX}
Y(t) \overset{(d)}{=} X^{[0,1]}(t) + t \tilde Y(1).
\end{equation}
From this it follows that the tails of $X^{[0,1]}(t)$ are dominated by those of $Y(t)$. Indeed, for any fixed $t\in[0,1]$ fix a constant $c \in \R$ such that $\P( t \tilde Y(1) \ge c ) > 0$ and $\P( t \tilde Y(1) \le c ) >0$. Then for all $x>0$,
$$
\P( Y (t) \ge x +c) \ge \P(X^{[0,1]}(t) \ge x ) \P( t \tilde Y(1) \ge c)
$$
so that
$$
\P(X^{[0,1]}(t) \ge x )  \le \frac1{\P( t \tilde Y(1) \ge c)} \P( Y (t) \ge x +c).
$$
This means that the right tail of $X^{[0,1]}(t)$ is at most a constant times that of $Y(t)$, which is Gaussian as $Y$ is a multiple of Brownian motion. A similar argument can be made for the left tail of $X^{[0,1]}(t)$. Hence we have proved that this random variable has moments of arbitrary order as claimed.

Furthermore, observe that by the domain Markov property and scaling,
	\begin{align*}X^{[0,T]}(1) &\overset{(d)}{=} \tilde{X}^{[0,S]}(1)+\frac{1}{S}X^{[0,T]}(S) \\
& = \tilde{X}^{[0,S]}(1)+ \frac{1}{\sqrt{S}}X^{[0,T/S]}(1)
	\end{align*}
	for any $T\ge S$, where $\tilde{X}$ and $X$ are independent. This implies $\var X^{[0,T]}(1)$ (which is well defined by the above) is an increasing function of $T$. Moreover, referring back to \eqref{eqn::domain_m_decomp}, we see that this variance is uniformly bounded and hence $ \var X^{[0,T]}(1)$ converges to a limit {as $T\to \infty$}: call it $s^2$.  By \eqref{relYX}, $\E[X^{[0,1]}(t)] = 0$, and so using  \eqref{eqn::domain_m_decomp} and the same argument again, we see that in fact the fourth moment of $X^{[0, 2^n]}(1)$ is bounded in $n$. Hence $\E[X^{[0,2^{n(k)}]}(1)^2]$ converges to $\E[ Y(1)^2] = \sigma^2$, but this limit must also be $s^2 = \lim_{T \to \infty } \E[X^{[0,T]}(1)^2]$ and so cannot depend on the subsequence $n(k)$. This means that the subsequential limit $Y$ does not depend on the subsequence, and hence the lemma is proved.
	
\end{proof}
In particular, an important consequence of this convergence is the following corollary:

\begin{corollary}\label{C:contmod}
   $X^{[0,1]}$ has a continuous modification.
\end{corollary}

\begin{proof}
By Lemma \ref{lem:local_BM}, the limit $Y$ in distribution as $n\to \infty $ of $(X^{[0,1]}(t)+tX^{[0,2^n]}(1))_{t\in [0,1]}$ is a multiple of Brownian motion and so has a continuous modification. Since the first of the two summands is simply $X^{[0,1]}$ and does not depend on $n$, and since the second summand is in the limit a.s. a linear function, we deduce that $X^{[0,1]}$ has a continuous modification.
\end{proof}

Now the main idea is to use the following change of variables which turns a Brownian motion to a Brownian bridge:

\begin{lemma}\label{lem:bb}
Let $(X(t))_{t \in [0,1]}$ be a process defined by $X(1) :=0$ and
$$
X(t) := (1-t) Z\left(\frac{t}{1-t}\right) \quad ;\quad t \in [0,1)
$$
where $(Z(s))_{s \in [0,\infty)}$ is a standard Brownian motion. Then $(X(t))_{t \in [0,1]}$ is a standard Brownian bridge on $[0,1]$.
\end{lemma}

This elementary and standard lemma can easily be verified by checking that the covariance of $X$ agrees with that of a Brownian bridge (and observing that $X$ retains the Gaussian character of $Z$).

We can now start the proof of Theorem \ref{thm:BB}.

\begin{proof}[Proof of Theorem \ref{thm:BB}]
First of all, note that by the scaling relation it is enough to prove the theorem for $I= [0,1]$. Consider the process
$$
W(t) := (1+t)X^{[0,1]}\left(\frac{t}{1+t}\right) \quad ; \quad t \in [0,\infty).
$$
In view of Lemma \ref{lem:bb} it suffices to show that $W$ is a multiple of Brownian motion.

We first claim that $(W(t))_{t\ge 0}$ {has independent increments}. Indeed, note that from the domain Markov property (applied to the interval $[s/(1+s), 1] \subset [0,1]$), we have for all $s<t$:
$$X^{[0,1]}(\frac{t}{1+t})  = X^{[0,1]}\left(\frac{s}{1+s}\right)\frac{1+s}{1+t}  + \tilde X^{[s/(1+s),1]}\left(\frac{t}{1+t}\right),$$
where $\tilde X^{[s/(1+s),1]}$ is a copy of $X^{[s/(1+s),1]}$ that is independent of $(X^{[0,1]}(u), u \le s/(1+s))$.
Then the above equation implies that
\begin{equation}
\label{eqn::wii}
W(t )  = W(s)+  (1+t) \tilde X^{[s/(1+s),1]} \left(\frac{t}{t+1}\right) \end{equation}
which proves {the claim}, since $\tilde X$ is independent of $(W(u), u \le s)$.

Now observe that by Corollary \ref{C:contmod}, $W$ admits a continuous modification. Moreover, since $X^{[0,1]}$ has zero mean (as already observed in the proof of Lemma \ref{lem:local_BM}), we see that $W$ is a martingale.

Finally, note that by \eqref{eqn::wii} and translation/scaling invariance, we can write
$$\E[(W(t) - W(s))^2] = (t-s) \frac{1+t}{1+s}\E[X^{[0,\frac{1+t}{t-s}]}(1)^2].
$$
We have already noted in the proof of Lemma \ref{lem:local_BM} that $\var(X^{[0,T]}(1))$ increases towards $\sigma^2$ as $T \to \infty$.
Hence,
letting $s \to t $, we obtain that
$$\E[ (W(t) - W(s))^2] \sim \sigma^2 (t-s)
$$
in the sense that the ratio of the two sides tends to 1 as $s \to t \ge 0$. Since $W$ has independent increments, we conclude that the quadratic variation of the continuous modification of $W$ is given by
$$ \langle W\rangle_t = \sigma^2 t.$$
Moreover we have $W(0) = 0$ a.s. Therefore, by L\'evy's characterisation of Brownian motion, we see that \[(W(t))_{t\geq 0} \eqd \sigma (B(t))_{t\geq 0},\] where $(B(t))_{t\geq 0}$ is a standard Brownian motion and $\sigma$ is the constant from Lemma \ref{lem:local_BM}. Thus
$$
X^{[0,1]}(t) \eqd \sigma(1-t)B\left(\frac{t}{1-t}\right) \quad ; \quad t \in  [0,1]
$$
which is an equivalent definition of a constant $\sigma$ times a Brownian Bridge in $[0,1]$ by Lemma \ref{lem:bb}.
\end{proof}

\section{Open problems}
\label{S:problems}

We end this article with a few open questions raised by our results. The most obvious ones are the following two:

\begin{problem}\label{Problem:moments}
  Is Theorem \ref{thm::characterisation_gff} true assuming only $\E[(h, \ph)^2] < \infty$ instead of $\E[(h, \ph)^4]< \infty$? Is it true without any moment assumption at all?
  \end{problem}

\begin{problem}\label{prob: high dimensions}
  Does an analogue of Theorem \ref{thm::characterisation_gff} hold in dimensions $d\ge 3$ (and if so, under what natural assumptions)?
\end{problem}

For Problem \ref{Problem:moments}, we believe that no moment assumptions (or perhaps only very weak moment assumptions) are necessary for the theorem to hold. In this direction, we were able to prove that certain averages of the field are Gaussian with moments assumption no stronger than Theorem \ref{thm:BB}. This is the case if we consider  a realisation of the It\^o \emph{excursion} measure in the upper half plane starting from zero (i.e., a process whose real coordinate is a Brownian motion, and whose imaginary coordinate is a sample from one-dimensional It\^o measure), and consider the hitting distribution by this process of a semicircle of radius $r$ centered at zero. Equivalently, this is the derivative at zero of the the harmonic measure on a semi-circle of radius $r$ centered at zero. Indeed, it can be shown that the field integrated against this measure is a time-change of Brownian motion (as a function of the radius). This is because there are martingale, Markovian properties together with scaling properties, which are sufficient to characterise Brownian motion. While this argument is very suggestive that no moments assumptions are needed, we could not exploit this (and so have chosen not to include a proof).


\medskip This makes it likely that no heavy-tailed analogue version of the GFF can exist if we insist on conformal invariance. Nevertheless it is interesting to try and investigate what are natural analogues (if any) of the GFF such that the integral against test function gives a heavy-tailed random variable.

\begin{problem}\label{Problem:stable}
  Does there exist a ``natural" stable version of the GFF?
\end{problem}

Let us give more details about what we mean in this question. In this paper, the domain Markov property is formulated in terms of harmonic functions, but in the context of Problem \ref{Problem:stable} it seems clear the notion of Markov property needs to be changed.
Indeed, one might hope that by adapting this definition of this hypothetical process to the one-dimensional case, one would recover the bridge of a stable L\'evy process, about which very little seems in fact to be known in general (see e.g. the recent paper \cite{ChaumontUB} for some basic properties). In particular, there does not seem to be an explicit relation between a stable bridge from 0 to 0 of duration one and a stable bridge from $a$ to $b$ of same duration for arbitrary values of $a,b$. This suggests that if a natural stable version of a GFF exists, it may be characterised by a more complex Markov property.

A natural way to ask the question precisely would be to try and discretise the problem, by considering the Ginzburg--Laundau $\nabla \varphi$ interface model. That is, for a domain $D \subset \C$, consider $D^{\delta}$ a fine mesh lattice approximation of $D$. On $D^\delta$, consider the random function $h^\delta$ defined on the vertices of $D^\delta$ through the law
$$
\P(  dh^\delta) \propto \prod_{x \sim y} \Phi\left( h^\delta(x) - h^\delta(y)\right) \prod_{x \in D^\delta } d h^\delta (x)
$$
where $\prod_x dh(x)$ is the product Lebesgue measure on $\R$ for all vertices in the graph, and $V$ is some fixed symmetric nonnegative function which decays to zero sufficiently fast that the total mass of the measure is finite. A priori this only defines a law up to a global additive constant, which can be fixed by requiring $h^\delta(x_0) = 0$ at some fixed vertex $x_0 \in D^\delta$. Then the question is to identify the limit (if it exists) as $\delta \to 0$ of the height function $h^\delta$, extended in some natural way to all of $D$ and viewed as a random distribution on $D$. Moreover, one can ask how the limit depends on the choice of $\Phi$. When $\Phi$ decays very fast at infinity (say if $\Phi$ is supported on a bounded interval) it is expected -- but not proved -- that the limit is a Gaussian free field. This is currently known only in the case where we can write $\Phi = e^{- V}$ for $V$ uniformly convex and $V''$  a Lipschitz function: see Miller \cite{MillerGL}, who relied on earlier work of Giacomin, Olla and Spohn \cite{GOS} and Naddaf and Spencer \cite{NS} for the analogous result in the full plane. However the case of bounded support remains wide open at the moment. To formulate the above problem concretely, we ask what happens when $\Phi$ is heavy-tailed: in particular, does the limit as $\delta \to 0$ exist? If so, what sort of Markov property does it satisfy?

\medskip In another direction, it is not entirely clear how to characterise other versions of the GFF in a similar way. For instance:

\begin{problem}
  What is the analogue of Theorem \ref{thm::characterisation_gff} for a GFF with free boundary conditions?
\end{problem}

(See e.g. \cite{LQGNotes} for a definition of the GFF with free (or Neumann) boundary conditions.)

\medskip Another natural family of random fields which arises naturally are the so-called fractional Gaussian fields (FGF for short), see \cite{lodhia2016} for a definition and survey of basic properties. Roughly, they are defined as $(-\Delta)^{-s/2} W$ where $W$ is white noise on $\R^d$, and $(-\Delta)^{-s/2}$ is the fractional Laplacian for a given $s \in \R$. By contrast with the hypothetical ``stable" GFF discussed above, FGFs can be seen as Gaussian free fields with long range interactions (see section 12.2 in \cite{lodhia2016}).
This includes the Gaussian free field (corresponding to $s=1$) and many other natural Gaussian fields. It turns out that FGFs enjoy a Markov decomposition similar to that of the GFF, where the notion of harmonic function is replaced by the notion of $s$-harmonic function (i.e., harmonic with respect to the fractional Laplacian $(-\Delta)^s$, see Proposition 5.4 in \cite{lodhia2016}). However, note that the fractional Laplacian is a nonlocal operator so this Markov decomposition is not a Markov property in the usual sense: the conditional law of the field given the values outside of some domain $U$ depend on more than just the boundary values.

In dimension two, FGFs are not conformally invariant at least in the sense of this article except if $s = 1$, since in general for a given $a \in \R$, we have $h(ax) = a^{s- d/2} h(x)$ in distribution (see below (3.4) in \cite{lodhia2016}). Nevertheless, it is natural to ask:

\begin{problem}
What properties characterise fractional Gaussian fields for a given $s \in \R$ ?
\end{problem}

\medskip Finally,
it is natural to ask what can be said on a given Riemann surface. In this case, the field $h$ should also have an ``instanton'' component, which describes the amount of height that one picks up as one makes a noncontractible loop over the surface. It is natural to allow this quantity to be nonzero in general, and to depend only on the equivalence class of the loop (for the homotopy relation). In the language of forms, this means that $\nabla h$ will be a closed one-form but not exact.

Characterising conformally invariant random fields with a natural Markov property would be particularly interesting because (a) there exists more than one natural field in this context (e.g., there is at least the standard Gaussian free field with mean zero as well as the so-called \emph{compactified GFF} which arises as the scaling limit of the dimer model on the torus, see \cite{dubedat_torsion} and \cite{BLRtorus}); and (b) in the context of the dimer model, there are natural situations (see again \cite{BLRtorus}) where a conformally invariant scaling limit is obtained but its law is unknown. Hence it would be of great interest to prove an analogue of Theorem \ref{thm::characterisation_gff} in the context of Riemann surfaces.

\begin{problem}
 Characterise fields on a given Riemann surface (including an instanton component) which enjoy a domain Markov property and conformal invariance.
\end{problem}

\small

\bibliographystyle{abbrv}
\bibliography{RGM}

\begin{thebibliography}{10}

\bibitem{DOZZ}
R.~R. A.~Kupianen and V.~Vargas.
\newblock The dozz formula from the path integral.
\newblock {\em J. {H}igh {E}nergy {P}hys.}, (5), 2018.

\bibitem{ahlfors}
L.~V. Ahlfors.
\newblock {\em Complex Analysis: An Introduction to the Theory of Analytic
  Functions of One Complex Variable}.
\newblock McGraw-Hill, 3rd edition, 1979.

\bibitem{BTLS}
J.~Aru, A.~Sep{\'u}lveda, and W.~Werner.
\newblock On bounded-type thin local sets of the two-dimensional gaussian free
  field.
\newblock {\em Journal of the Institute of Mathematics of Jussieu}, pages
  1--28, 2017.

\bibitem{LQGNotes}
N.~Berestycki.
\newblock Introduction to the {G}aussian free field and {L}iouville quantum
  gravity.
\newblock {\em
  http://www.statslab.cam.ac.uk/$\sim$beresty/Articles/oxford5.pdf}, 2015.

\bibitem{BLRdimers}
N.~Berestycki, B.~Laslier, and G.~Ray.
\newblock Universality of fluctutations in the dimer model.
\newblock {\em ArXiv preprint arXiv:1603.09740}, 2016.

\bibitem{BLRtorus}
N.~Berestycki, B.~Laslier, and G.~Ray.
\newblock Universality of the dimer model on {R}iemann surfaces, {I}.
\newblock {\em In preparation}, 2017.

\bibitem{billingsley}
P.~Billingsley.
\newblock {\em Probability and Measure}.
\newblock John Wiley and Sons, New York, 3rd edition, 1995.

\bibitem{CamiaGarbanNewman1}
F.~Camia, C.~Garban, and C.~M. Newman.
\newblock Planar {I}sing magnetization field {I}. {U}niqueness of the critical
  scaling limit.
\newblock {\em Ann. {P}robab.}, 43(2):528--571, 2015.

\bibitem{CamiaGarbanNewman2}
F.~Camia, C.~Garban, and C.~M. Newman.
\newblock Planar {I}sing magnetization field {II}. {P}roperties of the critical
  and near-critical scaling limits.
\newblock {\em Ann. Inst. Henri Poincar\'{e} (B)}, 52(1):146--161, 2016.

\bibitem{ChaumontUB}
L.~Chaumont and G.~U. Bravo.
\newblock Markovian bridges: weak continuity and pathwise constructions.
\newblock {\em Ann. Probab.}, 39(2):609--647, 2011.

\bibitem{Dubedat}
J.~Dub{\'e}dat.
\newblock {SLE} and the free field: partition functions and couplings.
\newblock {\em Journal of the {AMS}}, 22(4):995--1054, 2009.

\bibitem{dubedat_torsion}
J.~Dub{\'e}dat.
\newblock Dimers and families of {C}auchy {R}iemann operators {I}.
\newblock {\em Journal of the {AMS}}, 28(4):1063--1167, 2015.

\bibitem{DubedatGheissari}
J.~Dub{\'e}dat and R.~Gheissari.
\newblock Asymptotics of height change on toroidal {T}emperleyan dimer models.
\newblock {\em Journal of Statistical Physics}, 159(1):75--100, 2015.

\bibitem{DuplantierMillerSheffield}
B.~Duplantier, J.~Miller, and S.~Sheffield.
\newblock Liouville quantum gravity as a mating of trees.
\newblock {\em ArXiv preprint arXiv:1409.7055}, 2014.

\bibitem{DuplantierSheffield}
B.~Duplantier and S.~Sheffield.
\newblock Liouville quantum gravity and {KPZ}.
\newblock {\em Invent. Math.}, 185(2):333--393, 2011.

\bibitem{evans}
L.~Evans.
\newblock {\em {P}artial {D}ifferential {E}quations}.
\newblock Graduate Studies in Mathematics, Volume 19. AMS, 1998.

\bibitem{DKRV}
R.~R. F.~David, A.~Kupianen and V.~Vargas.
\newblock Liouville quantum gravity on the riemann sphere.
\newblock {\em Comm. {M}ath. {P}hys.}, 342(3):869--907, 2016.

\bibitem{cue}
Y.~Fyodorov, B.~A. Khoruzhenko, and N.~Simm.
\newblock On the characteristic polynomial of a random unitary matrix.
\newblock {\em Comm. {M}ath. {P}hys.}, 220(2):429--251, 2001.

\bibitem{gue}
Y.~Fyodorov, B.~A. Khoruzhenko, and N.~Simm.
\newblock Fractional {B}rownian motion with {H}urst index $h=0$ and and the
  {G}aussian unitary ensemble.
\newblock {\em Ann. Probab.}, 44(4):2980--3031, 2016.

\bibitem{GOS}
G.~Giacomin, S.~Olla, and H.~Spohn.
\newblock Equilibrium fluctuations for $ \nabla \varphi$ interface model.
\newblock {\em Ann. {P}robab.}, pages 1138--1172, 2001.

\bibitem{Kenyon_GFF}
R.~Kenyon.
\newblock Dominos and the {G}aussian free field.
\newblock {\em Ann. Probab.}, 29(3):1128--1137, 2001.

\bibitem{krantz}
S.~G. Krantz.
\newblock {\em Geometric function theory: explorations in complex analysis}.
\newblock Birkhauser, Boston, 2006.

\bibitem{Li}
Z.~Li.
\newblock Conformal invariance of isoradial dimers.
\newblock {\em arXiv preprint arXiv:1309.0151}, 2013.

\bibitem{lodhia2016}
A.~Lodhia, S.~Sheffield, X.~Sun, and S.~S. Watson.
\newblock Fractional gaussian fields: a survey.
\newblock {\em Probability Surveys}, 13:1--56, 2016.

\bibitem{MansuyYor}
R.~Mansuy and M.~Yor.
\newblock Harnesses, {L}{\'e}vy bridges and {M}onsieur {J}ourdain.
\newblock {\em Stoch. {P}roc. {A}ppl.}, 115(2):329--338, 2005.

\bibitem{MillerGL}
J.~Miller.
\newblock Fluctuations for the {G}inzburg-{L}andau interface model on a bounded
  domain.
\newblock {\em Comm. {M}ath. {P}hys.}, 308(3):591--639, 2011.

\bibitem{LQGandTBMI}
J.~{Miller} and S.~{Sheffield}.
\newblock {Liouville quantum gravity and the {B}rownian map {I}: The
  {QLE}(8/3,0) metric}.
\newblock {\em ArXiv preprint arXiv:1507.00719}, 2015.

\bibitem{NestingFieldCLE}
J.~Miller, S.~S. Watson, and D.~B. Wilson.
\newblock The conformal loop ensemble nesting field.
\newblock {\em Probab. Theory and Related Fields}, 163(3-4):769--801, 2015.

\bibitem{NS}
A.~Naddaf and T.~Spencer.
\newblock On homogenization and scaling limit of some gradient perturbations of
  a massless free field.
\newblock {\em Comm. {M}ath. {P}hys.}, 183(1):55--84, 1997.

\bibitem{polyakovstrings}
A.~M. Polyakov.
\newblock Quantum geometry of bosonic strings.
\newblock {\em Phys. Lett. B}, 103(3):207--210, 1981.

\bibitem{PowellWu}
E.~Powell and H.~Wu.
\newblock Level lines of the {G}aussian free field with general boundary data.
\newblock {\em Ann. Inst. Henri Poincar\'{e} (B)}, 53(4):2229--2259, 2017.

\bibitem{RiderVirag}
B.~Rider and B.~Vir{\'a}g.
\newblock The noise in the circular law and the {G}aussian free field.
\newblock {\em International Mathematics Research Notices}, 2007, 2007.

\bibitem{rudin}
W.~Rudin.
\newblock {\em {F}unctional {A}nalysis}.
\newblock McGraw Hill, 2nd edition, 1991.

\bibitem{schrammcharacterisation}
O.~Schramm.
\newblock Scaling limits of loop-erased random walks and uniform spanning
  trees.
\newblock {\em Israel J. Math}, 118:221–288, 2000.

\bibitem{SchrammSheffieldDiscrete}
O.~Schramm and S.~Sheffield.
\newblock Contour lines of the two-dimensional discrete {G}aussian free field.
\newblock {\em Acta {M}athematica}, 202(1):21--137, 2009.

\bibitem{SchrammSheffield}
O.~Schramm and S.~Sheffield.
\newblock A contour line of the continuum {G}aussian free field.
\newblock {\em Probab. Theory Related Fields}, 157(1-2):47--80, 2013.

\bibitem{Sheffield}
S.~Sheffield.
\newblock Gaussian free fields for mathematicians.
\newblock {\em Probab. Theory Related Fields}, 139(3-4):521--541, 2007.

\bibitem{zipper}
S.~Sheffield.
\newblock Conformal weldings of random surfaces: {SLE} and the quantum gravity
  zipper.
\newblock {\em Ann. Probab.}, 44(5), 2016.

\bibitem{Williams_harness}
D.~Williams.
\newblock Some basic theorems on harnesses. {S}tochastic analysis (a tribute to
  the memory of {R}ollo {D}avidson), 1973.

\end{thebibliography}

\bigskip
\noindent
{\sc Nathanael Berestycki}, {\em Universit\"at Wien}\footnote{On leave from the University of Cambridge}, {\tt <nathanael.berestycki@univie.ac.at>}\\
{\sc Ellen Powell}, {\em ETH Z\"{u}rich}, {\tt <egpowell12@gmail.com>}\\
{\sc Gourab Ray}, {\em University of Victoria}, {\tt <gourabray@uvic.ca>} \\

\end{document}